\documentclass[a4paper,10pt]{article}

\usepackage{amsmath}
\usepackage{amsthm}
\usepackage{amssymb}

\usepackage{sectsty}
\chapternumberfont{\sffamily\bfseries\LARGE\sectionrule{3ex}{3pt}{-1ex}{1pt}}
\sectionfont{\rmfamily\bfseries\sectionrule{3ex}{0pt}{-1ex}{1pt}}
\subsectionfont{\rmfamily\mdseries\large}
\subsubsectionfont{\rmfamily\mdseries}

\usepackage{mathpazo}

\allowdisplaybreaks

\usepackage{tikz}
\usetikzlibrary{decorations.markings}
\usetikzlibrary{matrix}

\usepackage{etex}

\usepackage{graphicx}
\usepackage{mathrsfs}
\usepackage{color}
\usepackage[all]{xy}

\usepackage{ctable}
\usepackage{multirow}

\usepackage{microtype}

\newcommand{\R}{\mathbb R}
\newcommand{\Q}{\mathbb Q}

\newcommand{\N}{\mathbb N}

\newcommand{\CC}{\mathcal C}

\newcommand{\VV}{\mathcal{V}}

\newcommand{\EE}{\mathcal E}

\newcommand{\id}{{\mathsf{id}}}

\newcommand{\X}{\mathsf{X}}

\newcommand{\Rplus}{[0,\infty]}

\newcommand{\zerounit}{\mathbf{0}}
\newcommand{\oneunit}{\mathbf{1}}

\DeclareFontFamily{U}{mathb}{}
\DeclareFontShape{U}{mathb}{m}{n}{
  <5> <6> <7> <8> <9> <10>
  <10.95> <12> <14.4> <17.28> <20.74> <24.88>
  mathb10
  }{}
\DeclareSymbolFont{mathb}{U}{mathb}{m}{n}
\DeclareMathSymbol{\mminus}{\mathbin}{mathb}{"01}

\newcommand{\toppoint}{\top}
\newcommand{\bottompoint}{\bot}

\newcommand{\leftadjoint}{\dashv}

\newcommand{\colim}[2]{\operatorname{colim}_{#1}(#2)}
\renewcommand{\lim}[2]{\operatorname{lim}_{#1}(#2)}
\newcommand{\presheaf}[1]{\widehat{#1}}
\DeclareFontFamily{U}{mathx}{}
\DeclareFontShape{U}{mathx}{m}{n}{
  <5> <6> <7> <8> <9> <10>
  <10.95> <12> <14.4> <17.28> <20.74> <24.88>
  mathx10
  }{}
\DeclareSymbolFont{mathx}{U}{mathx}{m}{n}
\DeclareMathAccent{\widecheck}{0}{mathx}{"71}
\newcommand{\copre}[1]{\widecheck{#1}}

\newcommand{\HK}{\EE_{\mathrm{HK}}}

\DeclareMathOperator{\dd}{d}

\newcommand{\op}{{\mathrm{op}}}

\newcommand{\Set}{\mathsf{Set}}

\newcommand{\GMet}{\mathrm{GMet}}
\newcommand{\GMetCo}{\GMet_{\mathrm{cocomp}}}

\newcommand{\isomorphic}{\simeq}
\newcommand{\isometric}{\cong}

\renewcommand{\phi}{\varphi}
\renewcommand{\epsilon}{\varepsilon}

\DeclareMathOperator{\Aim}{Aim}

 \DeclareMathOperator{\Hom}{Hom}

 \DeclareMathOperator{\End}{End}

\DeclareMathOperator{\Ob}{Ob}

\DeclareMathOperator{\Yoneda}{\mathcal{Y}}
\DeclareMathOperator{\Fix}{Fix}

\newtheorem{thm}{Theorem}
\newtheorem*{introthm}{Theorem}
\newtheorem{prop}[thm]{Proposition}
\newtheorem{cor}[thm]{Corollary}
\newtheorem{lemma}[thm]{Lemma}

\theoremstyle{definition}

\newcommand{\definition}{\textbf}

\usetikzlibrary{decorations.markings}
\makeatletter
\tikzset{nomorepostaction/.code={\let\tikz@postactions\pgfutil@empty}}
\makeatother

\usepackage{hyperref}

\begin{document}
\title{Tight spans, Isbell completions and semi-tropical modules.}
\author{Simon Willerton}

\maketitle

\begin{abstract}
In this paper we consider generalized metric spaces in the sense of Lawvere and the categorical Isbell completion construction.  We show that this is an analogue of the tight span construction of classical metric spaces, and that the Isbell completion coincides with the directed tight span of Hirai and Koichi.  The notions of categorical completion and cocompletion are related to the existence of semi-tropical module structure, and it is shown that the Isbell completion (hence the directed tight span) has two different semi-tropical module structures.\end{abstract}

\tableofcontents

\section*{Introduction}
This paper grew out of a desire to understand whether the tight span of a metric space could be understood in terms of the enriched category theory approach to metric spaces.  This led to understanding a link between two apparently unrelated constructions of Isbell, namely the tight span of metric spaces and the Isbell completion of categories; this is turn led, via categorical completeness,  to connections with tropical algebra.  In this introduction the main ideas of Isbell completion, semi-tropical algebra and tight spans will be given.   The intention is that this paper should be readable by mathematicians interested in metric spaces or tropical algebra, without much category theory background, and to allow them to see how category theoretic methods give interesting insight in this case.

\subsection*{The Isbell completion of a generalized metric space}
Lawvere~\cite{Lawvere:MetricSpaces} observed that a metric space can be viewed as something similar to a category and that from that perspective there is a natural generalization --- generalized metric space --- which means a set $X$ with a `distance' function $\dd\colon X\times X \to [0,\infty]$ such that $\dd(x,x)=0$ and $\dd(x,y)+\dd(y,z)\ge \dd(x,z)$ for all $x,y,z\in X$, with no further conditions like symmetry imposed.  Generalized metric spaces can be thought of as `directed' metric spaces.  From a category theoretic point of view, generalized metric spaces are precisely $[0,\infty]$-enriched categories and so much of the machinery of category theory can be utilized to study them.  In this paper we will look at the `Isbell completion' for generalized metric spaces.  

The Isbell completion of an ordinary category seems to be partly folklore, I learnt of it from Tom Leinster~\cite{BartlettLeinster:nCafe}.  The idea of the Isbell completion is that one embeds the original category in a certain cateogry of `presheaves' on the category, thus obtaining an `extension' of the original category.  This can be generalized reasonably straight forwardly to enriched categories.  One example worthy of note here is the Dedekind-MacNeille completion of a poset.  A poset $P$ can be considered as a category enriched over `truth values' and taking the Isbell completion in this case gives the Dedekind-MacNeille completion of $P$ (see~\cite{Trimble:mathoverflow}); in the particular example of the rational numbers $\Q$ with the usual ordering $\le$, the Isbell completion is the set of Dedekind cuts, i.e.~the real numbers $\R$ with the usual ordering.

In the case of a generalized metric space $X$ the Isbell completion $I(X)$, defined%
\footnote{As this paper was being finished Dusko Pavlovic showed me that he had considered~\cite{Pavlovic:QuantitativeConceptAnalysis} the Isbell completion of `proxets' which are basically the same as generalized metric spaces; his context and motivation are rather different to those here, however.}
 in Section~\ref{Section:DefnIsbell}, has the form of a set of pairs of functions  $f,g\colon X\to [0,\infty]$ satisfying certain conditions that allow us to think of $(f,g)$ as a point in an extension of $X$ where we consider $f(x)$ as the distance to $(f,g)$ from $x$ and $g(x)$ as the distance to $x$ from $(f,g)$.  In other words, every point in the Isbell completion $I(X)$ is specified by its distance to and from every point of $X$.  It should be added though that $f$ and $g$ do actually determine one another, so that $I(X)$ can actually be thought of a subset of the space of functions on $X$, for instance by just considering $f$.  In category theoretic terms the Isbell completion is the invariant part of the Isbell adjunction between presheaves and op-co-presheaves on $X$.

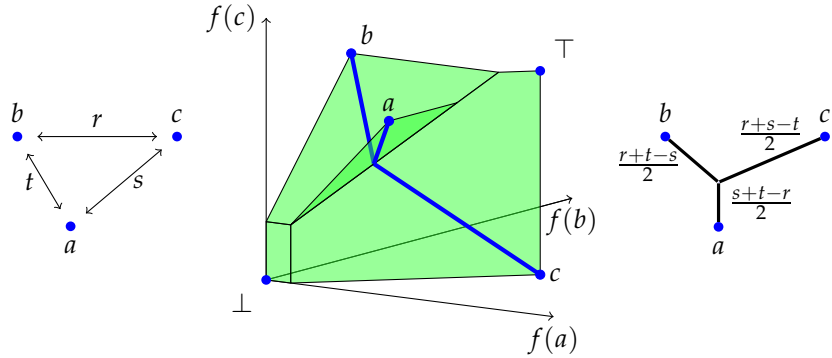
\begin{figure}[tb]
\begin{center}
  \begin{tikzpicture}[scale=0.7,baseline=-3cm)]
    \node [circle,draw=blue,fill=blue,thick, inner sep=0pt,minimum size=1mm,label=above:$b$](x) at (0,0) {};
  \node [circle,draw=blue,fill=blue,thick, inner sep=0pt,minimum size=1mm,label=above:$c$](y) at (3,0) {};
  \node [circle,draw=blue,fill=blue,thick, inner sep=0pt,minimum size=1mm ,label=below:$a$](z) at (1,-1.7) {};
 \begin{scope}[<->,shorten >=2mm,shorten <=2mm]
  \path (x) edge node[above] {$r$} (y);
  \path (z) edge node[pos=0.5,left] {$t$} (x);
  \draw (y) edge node[pos=0.5,right] {$s$} (z);
  \end{scope}
  \end{tikzpicture}
\begin{tikzpicture}[scale=0.4,line join=round]
\draw[style =thin,opacity=1,arrows=->](0,0)--(2*5.033,2*1.355) node[below ] {$f(b)$};
\draw[style =thin,arrows=->](0,0)--(1.3*7.27,1.3*-.938) node[below] {$f(a)$}; ;

\draw[style =thin,arrows=->](0,0)--(0,8.687) node[left] {$f(c)$};
\filldraw[draw=black,fill=green!80,fill opacity=0.5](0,0)--(.808,-.104)--(.808,1.826)--(0,1.931)--cycle;
\filldraw[draw=black,fill=green!80,fill opacity=0.5](.808,1.826)--(7.643,6.884)--(2.796,7.51)--(0,1.931)--cycle;
\filldraw[draw=black,fill=green!80,fill opacity=0.5](.808,1.826)--(7.643,6.884)--(9.01,6.931)--(9.01,.174)--(.808,-.104)--cycle;

\draw[draw=blue,style =ultra thick](2.796,7.51)--(3.542,3.849);

\filldraw[draw=black,draw opacity=1,fill=green!80,fill opacity=0.5](.808,1.826)--(6.276,5.873)--(4.039,5.27)--cycle;

\draw[draw=blue,style =ultra thick](4.039,5.27)--(3.542,3.849);
\draw[draw=blue,style =ultra thick](9.01,.174)--(3.542,3.849);

\draw [draw=blue,fill=blue] (4.039,5.27) circle (.15cm); 
\node at (4.039,5.27) [anchor=south] {$a$};
\draw [draw=blue,fill=blue] (2.796,7.51) circle (.15cm); 
\node at (2.796,7.51) [anchor=south west] {$b$};
\draw [draw=blue,fill=blue] (9.01,.174) circle (.15cm); 
\node at (9.01,.174) [anchor=west] {$c$};
\node [circle,draw=blue,fill=blue,thick, inner sep=0pt,minimum size=1mm ,label=above right:$\toppoint$](T) at (9.01,6.931){};
\node [circle,draw=blue,fill=blue,thick, inner sep=0pt,minimum size=1mm ,label=below left:$\bottompoint$](B) at (0,0) {};
\end{tikzpicture}
%
  \begin{tikzpicture}[scale=0.7,baseline=-3cm]
    \node [circle,draw=blue,fill=blue,thick, inner sep=0pt,minimum size=1mm,label=above:$b$](x) at (0,0) {};
  \node [circle,draw=blue,fill=blue,thick, inner sep=0pt,minimum size=1mm ,label=above:$c$](y) at (3,0) {};
  \node [circle,draw=blue,fill=blue,thick, inner sep=0pt,minimum size=1mm,label=below:$a$](z) at (1,-1.71) {};
  \node [circle,draw=blue,fill=blue,thick, inner sep=0pt,minimum size=0mm ]
    (o) at (1,-0.866) {};
  \draw [-,draw=blue,very thick] (o) edge node[pos=0.4,left] {$\tfrac{r+t-s}{2}$\;} (x) 
                edge  node[pos=0.5,above] {$\tfrac{r+s-t}{2}$} (y)
                edge  node[midway,right] {$\tfrac{s+t-r}{2}$} (z);
  \end{tikzpicture}
  \\
\end{center}
\caption{The classical metric space $A_{r,s,t}$, its Isbell completion (see Proposition~\ref{Prop:ArstCalculation}) 
and its tight span with the edge lengths marked.}
\label{Fig:ArstIsbellCompletion}
\end{figure}

For example, consider the generalized metric space of two points, with distances $r$ and $s$ from one to the other, its Isbell completion is an $r\times s$ rectangle with an asymmetric version of the $L_\infty$-metric on it.  This is pictured in Figure~\ref{Fig:ArsIsbellCompletion}.  Consider also a classical metric space with three points, its Isbell completion is the union of four parts of planes (Propostion~\ref{Prop:ArstCalculation}) and can be embedded in $\R^3$ with an asymmetric $L_\infty$-metric on it.  This is pictured in Figure~\ref{Fig:ArstIsbellCompletion}.  

In ordinary category theory there are notions of limits and colimits: we talk about categories being complete if they have all limits and cocomplete if they have all colimits; we talk about functors being continuous if they preserve limits and cocontinuous if they preserve colimits.  In enriched category theory there are corresponding notions, sometimes referred to as weighted limits and weighted colimits.  These specialize to notions of (categorical) limit and colimit for generalized metric spaces, although these should not be confused with usual notions of limit for metric spaces, this is just a coincidence of terminology, although you may refer to the work of Rutten~\cite{Rutten:WeightedColimitsMetricSpaces} if you want to see how they are related.  We show that the Isbell completion of a generalized metric space is both complete and cocomplete in this categorical sense: Theorem~\ref{Thm:FixRLComplete}, Theorem~\ref{Thm:FixLRComplete} and Theorem~\ref{Thm:YonedaCompleteCocomplete} combine to give the following.
\begin{introthm}
The Isbell completion $I(X)$ of a generalized metric space $X$ is both complete and cocomplete and the Yoneda embedding $X\to I(X)$ is both continuous and cocontinuous.
\end{introthm}
As alluded to above, an important role is played by presheaves, a presheaf is a function $f\colon X\to [0,\infty]$ satisfying a certain condition and the set of all presheaves forms a generalized metric space $\presheaf{X}$.  Abusing notation slightly, the Isbell completion $I(X)$ can be thought of as a subspace of the space of presheaves $\presheaf{X}$, via $\iota^1\colon I(X)\to \presheaf{X}$.  In fact there is a retraction $RL\colon \presheaf{X}\to I(X)$.  Colimits are very easy to calculate in $\presheaf{X}$ and it is possible to calculate a colimit in $I(X)$ by first including into $\presheaf{X}$, calculating the colimit and then retracting back to $I(X)$.  There is an analogous story with op-co-presheaves $\copre{X}^\op$, so there is an inclusion $\iota^2\colon I(X)\to \copre{X}^\op$ and a retraction $LR\colon\copre{X}^\op\to I(X)$ such that limits in $I(X)$ can be calculated using this.  This will be useful when we come to the semi-tropical module structures.

%
%
%

\subsection*{Semi-tropical algebra}
There has been interest from various directions in recent years in the area of `tropical mathematics' (also known as `idempotent' or `min-plus' mathematics).  This involves working with the ``tropical'' semi-ring consisting of  (possibly some variant of)  the extended real numbers $(-\infty,\infty]$ with $\min$ as the addition and $+$ as the multiplication.  This is a semi-ring as the addition does not have inverses.  In fact, the tropical semi-ring is a semi-field as all the elements, except for the additive unit $\infty$, have multiplicative inverses.
Here we are interested in what we will call the  `semi-tropical' semi-ring consisting of $[0,\infty]$ with $\min$ and~$+$.  Details are given in Section~\ref{Section:SemiTropical}.

A module over a semi-ring is a commutative monoid with an action of the semi-ring defined as you would define the action of a ring on an abelian group.  A module over the semi-tropical semi-ring $[0,\infty]$ will be called  a `semi-tropical module' and such a thing can be thought of being equipped with a `semi-flow': points can be moved forward in time by a positive amount but can not be moved back in time.

We are interested in considering a generalized metric space with a semi-tropical module structure that is compatible with the metric.  There are two compatibilities that are of interest and these give us the notions of metric semi-tropical module and co-metric semi-tropical module.  This can now be linked in to the category theory above.  Combining Theorem~\ref{Thm:CocompleteCometricModule} and Theorem~\ref{Thm:CompleteMetricModule} we get the following theorem.
\begin{introthm}
A skeletal generalized metric space is finitely complete if and only if it can be given the structure of a metric semi-tropical module.  

Similarly, a skeletal generalized metric space is finitely cocomplete if and only if it can be given the structure of a co-metric semi-tropical module.  
\end{introthm}

As we know that the Isbell completion of a generalized metric space is both complete and cocomplete we find (Corollaries~\ref{Cor:IsbellCoMetricModule} and~\ref{Cor:IsbellMetricModule}) that it is a semi-tropical module in two distinct ways, in a metric way --- with monoid multiplication $\boxplus$ and action $\boxdot$ --- and in a co-metric way --- with monoid multiplication $\oplus$ and action $\odot$.  We can illustrate this easily for the case of the two-point asymmetric metric space $N_{r,s}$ in Figure~\ref{Fig:BothModuleStructures}.  A more complicated example would be given by a three-point space as in Figure~\ref{Fig:ArstIsbellCompletion}.
  
\begin{figure}[ht]
\begin{center}
\begin{tikzpicture}[scale=1.5]
\draw[fill=green!80,fill opacity =0.5,thin] (0,0) rectangle (3,2);
\node [circle,draw=blue,fill=blue,thick, inner sep=0pt,minimum size=1mm ,label=above right:$b$](b) at (3,0) {};
\node [circle,draw=blue,fill=blue,thick, inner sep=0pt,minimum size=1mm ,label=above right:$a$](a) at (0,2) {};
\node [outer sep=0pt,inner sep=0pt] (p) at (1,1.2) {};
\node (q) at (2,0.5) {};
\node (pq) at (1,0.5) {};
\node (pqm) at (2,1.2) {};
\draw[postaction={nomorepostaction,decorate,
                    decoration={markings,mark=at position 0.5 with {\arrow{>}}}
                   },
                   draw=black, very thick] (p) -- +(0.8,0.8);
\draw[postaction={nomorepostaction,decorate,
                    decoration={markings,mark=at position 0.5 with {\arrow{>}}}
                   },
                   draw=black,  very thick] (p)+(0.8,0.8) -- (3,2);
\draw[postaction={nomorepostaction,decorate,
                    decoration={markings,mark=at position 0.5 with {\arrow{>}}}
                   },
                   draw=black, very thick] (p) -- +(-1,-1);
\draw[postaction={nomorepostaction,decorate,
                    decoration={markings,mark=at position 0.5 with {\arrow{>}}}
                   },
                   draw=black, very thick] (p)+(-1,-1) -- (0,0);
\node [circle,draw=blue,fill=blue,thick, inner sep=0pt,
minimum size=1mm, 
 ,label=above:{$p$}] at (p) {};
\node [circle,draw=blue,fill=blue,thick, inner sep=0pt,minimum size=1mm ,label=below:{$q$}] at (q) {};
\node [circle,draw=blue,fill=blue,thick, inner sep=0pt,minimum size=1mm ,label=below :{$p \oplus q$}] at (pq) {};
\node [circle,draw=blue,fill=blue,thick, inner sep=0pt,minimum size=1mm ,label=above :{$p \boxplus q$}] at (pqm) {};
\node (tauodotp) at (2,2.8) {$\tau\odot p$} ;
\draw[>=latex,->](tauodotp) -- (2.1,2.05);
\node (tauboxdotp) at (-0.5,1) {$\tau\boxdot p$} ;
\draw[>=latex,->](tauboxdotp) -- (.2,.47);
\end{tikzpicture}
\end{center}
\caption{The two semi-tropical module structures on $I(N_{r,s})$, with $\tau\in[0,\infty]$.}
\label{Fig:BothModuleStructures}
\end{figure}
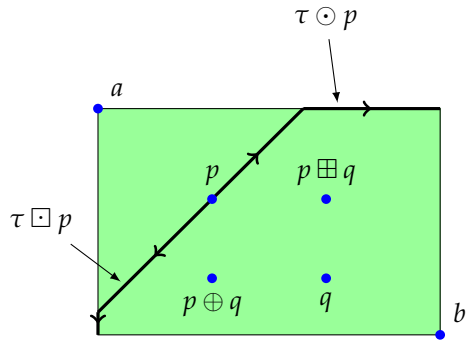

There is a way of computing these actions.  Recall that there is the inclusion $\iota^1\colon I(X)\to \presheaf X$, and the retraction $RL\colon \presheaf{X}\to I(X)$.  There is also a straightforward (co-metric) semi-tropical module structure on $\presheaf{X}$ defined pointwise, namely for $f,f'\in \presheaf{X}$ and $\tau\in [0,\infty]$
  \[(f\uplus f' )(x):= \min(f(x),g(x)),\qquad (\tau\star f)(x):=\tau+f(x).\]
The co-metric semi-tropical module structure on $I(X)$ is then calculated by including into $\presheaf{X}$, calculating the semi-tropical action and then retracting back onto $I(X)$.  Notationally we get
 \begin{align*}
 p\oplus q &:=RL(\iota^1 p \uplus \iota^1 q)
 &
\tau\odot p &:=  RL(\tau \star \iota^1 p).
 \end{align*}
The other, metric, semi-tropical module structure is similarly calculated using $\iota^2\colon I(X)\to \copre{ X}^\op$, and the retraction $LR\colon \copre{ X}^\op\to I(X)$.

\subsection*{The tight span}
Every classical metric space $\X$ has a `tight span' $T(\X)$.  This has been discovered independently on several occasions and goes by many names.  Isbell called it the injective envelope~\cite{Isbell:SixTheoremsInjective}, Dress called it the tight span~\cite{Dress:TreesTightExtensions} and Chrobak and Larmore called it the convex hull~\cite{ChrobakLarmore}.  Two good introductions to the theory of tight spans are \cite{DressEtAl:TTheoryOverview} and \cite{Epstein:HyperconvexityTalk}.  One naive way of thinking of the tight span is that it is a `small' contractible metric space in which the metric space embeds isometrically, but the full story is somewhat richer than that.  An example of the three-point metric space is pictured  in Figure~\ref{Fig:ArstIsbellCompletion}.

In general the tight span of a finite metric space is a finite cell complex of dimension at most half the number of points in $\X$, i.e.~$\dim T(\X)\le \frac{1}{2} \# \X$.
A metric space embeds into a tree if and only if the tight span is a tree, because of this the tight span has become a useful tool in phylogenetic analysis~\cite{DressEtAl:TTheoryOverview}.  The tight span also has applications in group cohomology~\cite{Dress:TreesTightExtensions}, server placement on a network~\cite{ChrobakLarmore} and multicommodity flow~\cite{Karzanov:Minimum0Extensions}.  Develin and Sturmfels~\cite{DevelinSturmfels:Tropical,DevelinSturmfelsErrata} showed connections with tropical mathematics, showing in some cases that the tight span was related to the tropical hull of the metric.

The tight span is related to the Isbell completion in the following way (Theorem~\ref{Thm:TightSpanInIsbell}).  For a classical metric space $\X$ the tight span $T(\X)$ is the largest subset of of the Isbell completion $I(\X)$ which contains $\X$ and for which the restriction of the generalized metric is actually a classical metric.

Having understood the connection between Isbell's tight span and the Isbell completion, I discovered that, motivated by applications in multicommodity flow, Hirai and Koichi~\cite{HiraiKoichi:DirectedDistances} had constructed an analogue of the tight span for what they called finite `directed' metric spaces, by which they meant a space with the structure of a classical metric space, except that the symmetry axiom $\dd(x,y)=\dd(y,x)$ is not imposed.  For a directed metric space $X$ we will denote their `directed tight span' by $\HK(X)$.  In Theorem~\ref{Thm:HiraiKoichi} we see that for directed metric spaces the directed tight span and the Isbell completion coincide: $\HK(X)\isometric I(X)$.  From the semi-tropical module structures described above, we immediately get the following.
\begin{introthm}
For a directed metric space, the directed tight span of Hirai and Koichi can be given two canonical semi-tropical module structures.
\end{introthm}

\subsection*{Acknowledgements}
This paper was inspired by a short conversation at the $n$-Category Caf\'e~\cite{BartlettLeinster:nCafe}; I would like to thank Bruce Bartlett, Tom Leinster and Andrew Stacey for their contributions.

 
\section{Metric spaces as enriched categories, briefly}
\label{Section:MetricSpaces}
In this section we give a give introduction to the idea of viewing metric spaces as enriched categories.  This gives rise to Lawvere's notion of generalized metric space.  We give the relevant concepts from enriched category theory in this context, one of the most important concepts here being that of the generalized metric space of presheaves; a presheaf being the appropriate notion of scalar-valued function on a space.  References for this include~\cite{Lawvere:MetricSpaces,Kelly:EnrichedCategoryTheory, Borceux:Handbook2}.  We finish the section by mentioning adjunctions and state the idempotent property exhibited by metric space adjunctions.
 
Recall that a \definition{small category} $\CC$ consists of a set $\Ob(\CC)$ of objects together with the following data, satisfying the so-called `associativity' and `unit' axioms.
 \begin{enumerate}
 \item For each pair $c,c'\in \Ob(\CC)$ there is a \emph{set} of morphisms $\Hom_\CC(c,c')$.
 \item For each $c\in \Ob(\CC)$ there is an identity morphism; equivalently, there is a specified \emph{function} $\{\ast\}\to \Hom_\CC(c,c)$.
 \item For each triple $c,c',c''\in \Ob(\CC)$ there is a \emph{function}, known as \definition{composition}, $\Hom_\CC(c,c')\times\Hom_\CC(c',c'')\to \Hom_\CC(c,c'')$.
 \end{enumerate}

 The definition relies on the category of sets with its Cartesian product $\times$ and the unit object $\{\ast\}$ for this Cartesian product.  It turns out that for a category $\VV$ which in similar to $\Set$ in that it has a tensor product with a unit object, we can define the notion of a category enriched in $\VV$, or a $\VV$-category.  Whilst such a thing has a set of objects, every other instance of ``set'' and ``function'' in the above definition is replaced by ``object of $\VV$'' and ``morphism in $\VV$''.

   In particular, we can use the category $\Rplus$ which has as its objects the set $[0,\infty]$ of the non-negative reals together with infinity, and has a morphism $a\to b$ precisely if $a\ge b$.  The tensor product on $\Rplus$ is taken to be addition $+$ and so the unit object is $0$.  To tie-in more with standard metric space notation, for an $\Rplus$-category $X$ we will write the pair $(X,\dd_X)$ instead of $(\Ob(X),\Hom_X)$.  This means that an $\Rplus$-category, which we will call a \definition{generalized metric space}, consists of a set $X$ together with the following data.
 \begin{enumerate}
 \item For each pair $x,x'\in X$ there is a \emph{number} $\dd_X(x,x')\in [0,\infty]$.
 \item For each $x\in X$ we have the \emph{inequality} $0\ge \dd_X(x,x)$.
 \item For each triple $x,x',x''\in X$ we have the \emph{inequality} $\dd_X(x,x')+\dd_X(x',x'') \ge \dd_X(x,x'')$.
 \end{enumerate}
 It transpires that the associativity and unit conditions are vacuous in this case.  The second condition above can, of course, be more sensibly written as $\dd_X(x,x)=0$.

 It should be clear from this definition that a classical metric space is such a thing, hence the name {``generalized metric space''}.  The key difference from the definition of a classical metric space, however, is that symmetry of the metric is not imposed, so in general for $x,y\in X$ we have $\dd_X(x,y)\ne \dd_X(y,x)$, which is why such a thing can be thought of as a \emph{directed} metric space as  Hirai and Koichi \cite{HiraiKoichi:DirectedDistances} might say.  The other, less important differences, are that $\infty$ is allowed as a distance and that two distinct points can be a distance $0$ apart.

In a generalized metric space $X$, if two points $x,x'\in X$ are mutually a distance $0$ from each other, i.e.~$\dd(x,x')=0=\dd(x',x)$, then we say that they are \definition{isomorphic} and write $x\isomorphic x'$.  Isomorphic points cannot be distinguished by metric means.  The space $X$ is said to be \definition{skeletal} if there are no distinct points which are isomorphic, meaning $x\isomorphic x'$ implies $x=x'$.  By definition, classical metric spaces are skeletal. 

 The upshot of all this is that a lot of category-theoretic machinery can be applied to the theory of metric spaces; for one recent example see Leinster's definition of the magnitude of a metric space~\cite{Leinster:Magnitude}.

 We now come to the correct notion of function from this point of view.  The right notion of map between generalized metric spaces is what we will refer to as a {`short map'}, but which we could also call a `distance non-increasing map'; this is the translation to the case of metric spaces of the notion of enriched functor.  A \definition{short map} between generalized metric spaces is a function $f\colon X\to Y$ such that 
  \[\dd_X(x,x')\ge \dd_Y(f(x),f(x'))\quad\text{ for all }x,x'\in X.\]  
  
  For a generalized metric space $X$, the \definition{opposite} generalized metric space $X^\op$ is defined to be the generalized metric space with the same set of points, but with the generalized metric reversed: $\dd_{X^\op}(x,x'):=\dd_{X}(x',x)$.  Clearly, a classical metric space is equal to its opposite.

 Next there is the standard notion of functional (or scalar-valued function) on a generalized metric space; this is  a \definition{presheaf}, i.e.~a short map $f\colon X^\op \to \Rplus$, where $\Rplus$ is interpreted as a generalized metric space with the asymmetric metric, meaning that the set of points is $[0,\infty]$ and the metric, for $a,b\in[0,\infty]$, is given by $\dd_{\Rplus}(a,b):=\max(b-a,0)$.  As the last function is so useful we will write it as $ b\mminus a$ and call it the \definition{truncated difference}.  This all means that a presheaf on $X$ is set function $f\colon X\to [0,\infty]$ such that for all $x,x'\in X$ we have
   \[\dd_X(x,x')\ge f(x)  \mminus f(x').\]
We write $\presheaf X$ for the generalized metric space of all presheaves with the metric given by
  \[\dd_{\presheaf X}(f,g):=\sup_{x\in X } \left( g(x)\mminus f(x)\right).\]
 The generalized metric space $X$ maps isometrically in to its space of presheaves via  ``$x$ goes to the distance-to-$x$ functional'':
   \[
    \Yoneda\colon X\to \presheaf X;
     \qquad
     x\mapsto \dd_X({-},x).
  \]
This is the generalized metric space \definition{Yoneda map}.

There is similarly a generalized metric space $\copre X$ of \definition{co-presheaves}, that is short maps $X\to \Rplus$, which can be thought of as set maps $f\colon X \to [0,\infty]$ satisfying, for all $x,x'\in X$,
    \[\dd_X(x,x')\ge f(x')  \mminus f(x),\]
with the generalized metric given by
  \[\dd_{\copre X}(f,g):=\sup_{x\in X } \left( g(x)\mminus f(x)\right).\] 
However, the more appropriate space in this paper is the opposite generalized metric space $\copre X^\op$, which we can call the \definition{space of op-co-presheaves}.  We have the  \definition{co-Yoneda map} given by ``$x$ goes to the distance-from-$x$ function'':
   \[
     \Yoneda\colon X\to \copre X^\op;
     \qquad
     x\mapsto\dd_X(x,{-}).
  \]
If $X$ is a classical metric space then a presheaf is the same thing as a co-presheaf, so $\presheaf{X}\isometric \copre{X}$; however, $\presheaf{X}$ and $\copre{X}$ are not classical metric spaces, as they are not symmetric, so will be different to $\copre{X}^\op$ the space of op-co-presheaves.

It is worth noting here that the space $\presheaf X$ of presheaves and the space $\copre X^\op$ of op-co-presheaves are skeletal.  For instance, if we have presheaves $f,g\in \presheaf X$ then $\dd_{\presheaf X}(f,g)=0$ if and only if $f(x)\ge g(x)$ for all $x\in X$; so $f\isomorphic g$ implies that $f(x)= g(x)$ for all $x\in X$ and hence $f=g$.  If we start with a non-skeletal generalized metric space then the Yoneda map is not going to be injective, although it will be an isometry: isomorphic points in $X$ will map to the same presheaf.  The image of $X$ under the Yoneda map can thus be thought of as a skeletization of $X$.

We will now mention one further concept from enriched category theory which will be key in this paper.  An \definition{adjunction}, written $L\leftadjoint R$, is a pair of short maps $L\colon X\to Y$ and $R\colon Y\to X$ between generalized metric spaces $X$ and $Y$ such that 
  \[\dd_Y(L(x),y)=\dd_X(x,R(y))\quad \text{for all }x\in X,\ y\in Y.\]
An equivalent condition is that 
  \[\dd_Y(LR(y),y)=0\ \text{for all }y\in Y\quad\text{and}\quad  \dd_X(x,RL(x))=0\ \text{for all }x\in X.\]
It follows easily from this definition that, given such an adjunction, the composites $RL\colon X\to X$  and $LR\colon Y\to Y$ are both idempotent, meaning
\[RLRL(x)\isomorphic RL(x)\ \text{for all }x\in X\quad\text{and}\quad LRLR(y)\isomorphic LR(y)\ \text{for all }y\in Y.\]
As an aside, this means that all monads and comonads on generalized metric spaces are idempotent.

\section{Isbell's tight span in category-theoretic language}
In this section we state Isbell's original definition of the tight span, his `injective envelope', and then reformulate it in a fashion amenable to a category theoretic analysis.  We then compare the two approaches in some simple examples, seeing that the two approaches, but not the final answers, are indeed somewhat different.  Throughout this section we use a sans serif symbol such as $\X$ to denote a \emph{classical} metric space.  Whilst in the literature the tight span and the injective envelope are used to mean the same thing, here I will use the terminology to distinguish two isometric metric spaces.

\subsection{Two definitions of the tight span}
\label{Section:DefnOfTightSpan}
Isbell, in his 1964 paper \cite{Isbell:SixTheoremsInjective}, constructs for a classical metric space $\X$ the `injective envelope' $\EE(\X)$ in the following way.  Firstly define $\Aim(\X)$ the \definition{aim} of $\X$ by $f\in\Aim(\X)$ if $f$ is  a real-valued function on $\X$ which satisfies
  \[f(x)+f(y)\ge \dd(x,y)\quad \text{for all }x,y\in \X. \tag{$*$}\label{Eqn:comp}\]
A function $f\in\Aim(\X)$ is \definition{pointwise-minimal} if whenever $g\in\Aim(\X)$ satisfies  $g(x)\le f(x)$ for all $x\in \X$ then $g=f$.   We then define the \definition{injective envelope}
$\EE(\X)$ to be the classical metric space of pointwise-minimal functions in $\Aim(\X)$, with the metric on $\EE(\X)$ given by
  \[\dd_{\EE(\X)}(f,g)=\sup_{x\in \X}|f(x)-g(x)|.\]

That is the standard definition of the injective envelope or tight span, however, we can characterize it in the following way which is amenable to the category theoretic approach.  Let the \definition{tight span} $T(\X)\subset \presheaf \X$ be the subset of presheaves on $\X$ such that a presheaf $f\colon \X\to [0,\infty]$ is in $T(\X)$
if
 \[
   f(x)=\sup_{y \in \mathfrak{\X}}\bigl(\dd(x,y)\mminus f(y)\bigr) \quad \text{for all }x\in \X.
   \tag{$\dagger$}\label{Eqn:SymDuality}
  \]
The metric on $T(\X)$ is induced from the generalized metric on $\presheaf \X$.  A calculation shows that the metric on $T(\X)$ is actually symmetric.  

These two definitions are equivalent in that there is a canonical isomorphism of metric spaces $\EE(\X)\isometric T(\X)$: indeed, in some sense, they have the same set of points; however, they are defined as subspaces of different spaces.  The details of the proof that they are isometric are given by Dress~\cite{Dress:TreesTightExtensions};
the subtle part is showing that the metrics actually agree.   Theorem~\ref{Thm:HiraiKoichi} generalizes this result and the proof given there is a generalization of Dress' proof.

Before comparing the injective envelope $\EE(\X)$ and the tight span $T(\X)$ in some examples, it is worth noting two things about the tight span.  The first thing is that $\X$ maps isometrically in to its tight-span $T(\X)$ via the Yoneda map:
\[
  \Yoneda\colon \X\to T(\X); \qquad x\mapsto \dd({-},x).
\]
The second is that one can interpret the condition~\eqref{Eqn:SymDuality} as follows.  The idea is that a function $f$ represents a point, say $p_f$, in an extension of $\X$ with $\dd(p_f,x)=f(x)$ for all $x\in \X$.  The supremum condition~\eqref{Eqn:SymDuality} on $f$ is saying that ``for every point $x\in \X$ the point $p_f$ is arbitrarily close to being on a geodesic from $x$ to another point''; in other terms, for all $\epsilon \ge 0$ there is a $y$ such that $\dd(x,p_f)+\dd(p_f,y)\le d(x,y)+\epsilon$.  This is expressing the minimality of the tight-span $T(\X)$, saying that the points of $T(\X)$ must `lie between' the points of $\X$.

\subsection{Examples}
Here are two examples which compare the two approaches to the tight span.

The two-point metric space $A_{r}$ is pictured in  Figure~\ref{Fig:ArTightSpan}.  We define $\Aim(A_{r})\in \R^{\{a,b\}}$ by $f\in \Aim(A_{r})$ if $f(a)+f(b)\ge r$, $f(a)\ge 0$, and $f(b)\ge 0$.  The minimal functions are those satisfying $f(a)+f(b)=r$.  This defines the injective envelope $\EE(A_r)$ and is pictured in  Figure~\ref{Fig:ArTightSpan}.
 On the other hand the presheaf space $\presheaf {A_r}\subset [0,\infty]^{\{a,b\}}$ is defined by $f\in \presheaf {A_r}$ if $\left| f(a)-f(b)\right| \le r$.  This is also pictured in Figure~\ref{Fig:ArTightSpan} with the tight span also drawn.
 
The three-point classical metric space $A_{r,s,t}$ is also pictured in  Figure~\ref{Fig:ArTightSpan}.  Note that the injective envelope $\EE(A_{r,s,t})$ and the tight span $T(A_{r,s,t})$ are defined as subsets of different spaces but are isometric.

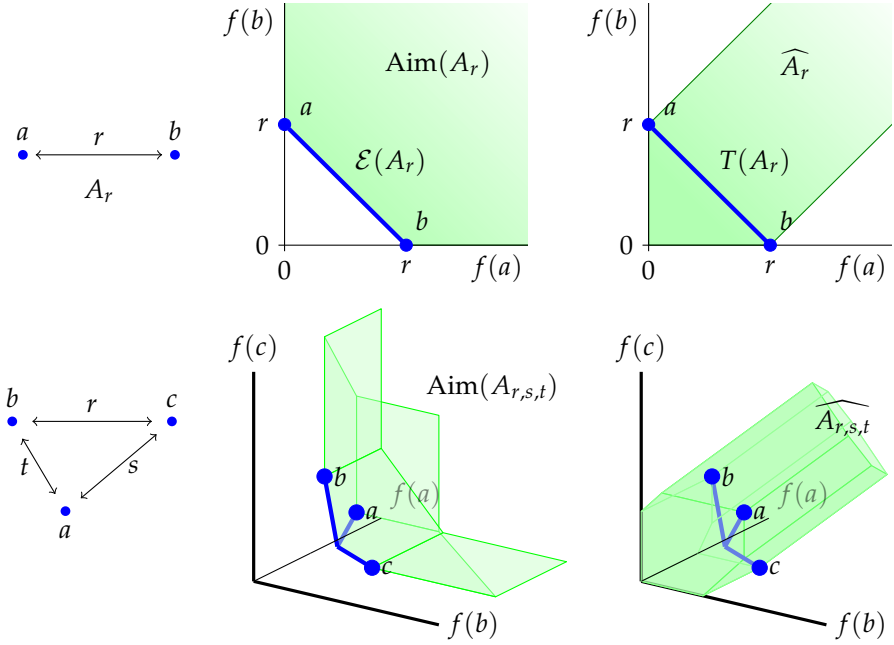
\begin{figure}[htb]
\centering
\begin{tikzpicture}[baseline=-2cm]
\node [circle,draw=blue,fill=blue,thick, inner sep=0pt,minimum size=1mm ,label=above:$b$](b) at (2,0) {};
\node [circle,draw=blue,fill=blue,thick, inner sep=0pt,minimum size=1mm ,label=above:$a$](a) at (0,0) {};
\draw[<->,shorten >=3pt,shorten <=3pt] (a) -- node[above]{$r$} (b);
\node at (1,-0.2) [anchor=north] {$A_r$};
\end{tikzpicture}
\quad
\begin{tikzpicture}[scale=0.8]
\clip (-1,-1) rectangle (4,4);
\shadedraw[left color=green!30,right color=white, draw=green!50!black, shading angle = 135]
(2,0) -- (5,0) -- (5,5) -- (0,5) -- (0,2) -- cycle;

\draw  (-0.1,0) -- (5,0);
\draw  (0,-0.1) -- (0,5);
\node at (3.5,0) [anchor=north] {$f(a)$};
\node at (0,3.7) [anchor=east] {$f(b)$};
\draw (2,0) -- (0,2);
\node [minimum size=0.1cm,inner sep = 0cm] (a) at ( 0,2){};
\draw [draw=blue,fill=blue] (a) circle (0.1cm);
\node [inner sep = 0cm] (b) at ( 2,0){};
\draw [draw=blue,fill=blue] (b) circle (0.1cm);
\draw [color=blue,ultra thick] (a) -- (b);
\node at (0,-0.1) [anchor = north] {$0$};
\node at (-0.1,0) [anchor = east] {$0$};
\node at (2,-0.1) [anchor = north] {$r$};
\node at (-0.1,2) [anchor = east] {$r$};
\node at (1,1) [anchor = south west] {$\EE(A_r)$};
\node at (1.5,3) [anchor = west] {$\Aim(A_r)$};
\node at (0.1,2) [anchor = south west] {$a$};
\node at (2,0.1) [anchor = south west] {$b$};

\end{tikzpicture}
\qquad
\begin{tikzpicture}[scale=0.8]
\clip (-1,-1) rectangle (4,4);
\shadedraw[left color=green!30,right color=white, draw=green!50!black, shading angle = 135]
(0,0) -- (2,0) -- (5,3) -- (5,5) -- (3,5) -- (0,2) -- cycle;

\draw (-0.1,0) -- (5,0);
\draw (0,-0.1) -- (0,5);
\node at (3.5,0) [anchor=north] {$f(a)$};
\node at (0,3.7) [anchor=east] {$f(b)$};

\node [minimum size=0.1cm,inner sep = 0cm] (a) at ( 0,2){};
\draw [draw=blue,fill=blue] (a) circle (0.1cm);
\node [inner sep = 0cm] (b) at ( 2,0){};
\draw [draw=blue,fill=blue] (b) circle (0.1cm);
\draw [color=blue,ultra thick] (a) -- (b);
\node at (0,-0.1) [anchor = north] {$0$};
\node at (-0.1,0) [anchor = east] {$0$};
\node at (2,-0.1) [anchor = north] {$r$};
\node at (-0.1,2) [anchor = east] {$r$};
\node at (1,1) [anchor = south west] {$T(A_r)$};
\node at (2,3) [anchor = west] {$\presheaf{A_r}$};
\node at (0.1,2) [anchor = south west] {$a$};
\node at (2,0.1) [anchor = south west] {$b$};
\end{tikzpicture}
\\
  \begin{tikzpicture}[scale=0.7,baseline=-3cm)]
    \node [circle,draw=blue,fill=blue,thick, inner sep=0pt,minimum size=1mm,label=above:$b$](x) at (0,0) {};
  \node [circle,draw=blue,fill=blue,thick, inner sep=0pt,minimum size=1mm,label=above:$c$](y) at (3,0) {};
  \node [circle,draw=blue,fill=blue,thick, inner sep=0pt,minimum size=1mm ,label=below:$a$](z) at (1,-1.7) {};
 \begin{scope}[<->,shorten >=2mm,shorten <=2mm]
  \path (x) edge node[above] {$r$} (y);
  \path (z) edge node[pos=0.5,left] {$t$} (x);
  \draw (y) edge node[pos=0.5,right] {$s$} (z);
  \end{scope}
  \end{tikzpicture}
  \quad
\begin{tikzpicture}[line join=round]
\filldraw[draw=green,fill=green!20,fill opacity=0.5](1.353,.914)--(1.353,2.457)--(2.436,2.201)--(2.436,.658)--cycle;
\filldraw[draw=green,fill=green!20,fill opacity=0.5](1.558,.18)--(1.101,.46)--(1.353,.914)--(2.436,.658)--(3.182,-.205)--cycle;
\filldraw[draw=green,fill=green!20,fill opacity=0.5](1.353,.914)--(1.101,.46)--(.932,1.391)--(.932,3.243)--(1.353,2.457)--cycle;
\draw[style =very thick](0,0)--(2.436,-.577);
\draw[draw=blue,style =ultra thick](1.353,.914)--(1.101,.46);
\draw[style =very thick](0,0)--(0,2.778);
\filldraw[draw=green,fill=green!20,fill opacity=0.5](1.558,.18)--(2.49,.645)--(4.114,.26)--(3.182,-.205)--cycle;
\draw(0,0)--(.746,.372);
\filldraw[draw=green,fill=green!20,fill opacity=0.5](.932,1.391)--(.932,3.243)--(1.678,3.615)--(1.678,1.763)--cycle;
\filldraw[draw=green,fill=green!20,fill opacity=0.5](.932,1.391)--(1.101,.46)--(1.558,.18)--(2.49,.645)--(1.678,1.763)--cycle;
\draw[draw=blue,style =ultra thick](.932,1.391)--(1.101,.46);
\draw[draw=blue,style =ultra thick](1.558,.18)--(1.101,.46);
\draw(.746,.372)--(1.678,.837);

        \node at (1.678,.837) [anchor=south west,opacity=0.5] {$f(a)$};
        \node at (2.436,-.577) [anchor=west] {$f(b)$};
        \node at (0,2.778) [anchor=south] {$f(c)$};
        \node at (2.149,2.578) [anchor=west]{$\Aim(A_{r,s,t})$};
        
        \draw [draw=blue,fill=blue] (1.353,.914) circle (.1cm); 
        \node at (1.353,.914) [anchor=west] {$a$};
        \draw [draw=blue,fill=blue] (.932,1.391) circle (.1cm); 
        \node at (.932,1.391) [anchor=west] {$b$};
        \draw [draw=blue,fill=blue] (1.558,.18) circle (.1cm); 
        \node at (1.558,.18) [anchor=west] {$c$};
        \end{tikzpicture}
{}
\quad
\begin{tikzpicture}[line join=round]
\draw[style =very thick](0,0)--(2.436,-.577);
\filldraw[draw=green!40,fill=green!20,fill opacity=0.5](3.234,1.418)--(2.761,2.124)--(2.456,2.516)--(2.252,2.633)--(2.725,1.927)--(3.029,1.535)--cycle;
\filldraw[draw=green,fill=green!60,fill opacity=0.5](1.353,.914)--(2.725,1.927)--(2.252,2.633)--(.271,1.171)--cycle;
\filldraw[draw=green,fill=green!60,fill opacity=0.5](1.353,.297)--(3.029,1.535)--(2.725,1.927)--(1.353,.914)--cycle;
\filldraw[draw=green,fill=green!60,fill opacity=0.5](0,0)--(0,.926)--(.271,1.171)--(1.353,.914)--(1.353,.297)--(.812,-.192)--cycle;
\draw[draw=blue,style =ultra thick](1.353,.914)--(1.101,.46);
\filldraw[draw=green,fill=green!60,fill opacity=0.5](.812,-.192)--(1.558,.18)--(3.234,1.418)--(3.029,1.535)--(1.353,.297)--cycle;
\draw[style =very thick](0,0)--(0,2.778);
\filldraw[draw=green!40,fill=green!20,fill opacity=0.5](0,.926)--(.271,1.171)--(2.252,2.633)--(2.456,2.516)--(.932,1.391)--cycle;
\filldraw[draw=green!40,fill=green!20,fill opacity=0.5](0,0)--(.746,.372)--(1.558,.18)--(.812,-.192)--cycle;
\filldraw[draw=green!40,fill=green!20,fill opacity=0.5](0,0)--(.746,.372)--(.932,.774)--(.932,1.391)--(0,.926)--cycle;
\draw(0,0)--(.746,.372);
\draw[draw=blue,style =ultra thick](.932,1.391)--(1.101,.46);
\draw[draw=blue,style =ultra thick](1.558,.18)--(1.101,.46);
\filldraw[draw=green!40,fill=green!20,fill opacity=0.5](3.234,1.418)--(2.761,2.124)--(.932,.774)--(.746,.372)--(1.558,.18)--cycle;
\filldraw[draw=green!40,fill=green!20,fill opacity=0.5](2.456,2.516)--(2.761,2.124)--(.932,.774)--(.932,1.391)--cycle;
\draw(.746,.372)--(1.678,.837);

        \node at (1.678,.837) [anchor=south west,opacity=0.5] {$f(a)$};
        \node at (2.436,-.577) [anchor=west] {$f(b)$};
        \node at (0,2.778) [anchor=south] {$f(c)$};
        \node at (2.149,2.161) [anchor=west]{$\presheaf{A_{r,s,t}}$};
        
\draw [draw=blue,fill=blue] (1.353,.914) circle (.1cm); 
\node at (1.353,.914) [anchor=west] {$a$};
\draw [draw=blue,fill=blue] (.932,1.391) circle (.1cm); 
\node at (.932,1.391) [anchor=west] {$b$};
\draw [draw=blue,fill=blue] (1.558,.18) circle (.1cm); 
\node at (1.558,.18) [anchor=west] {$c$};
\end{tikzpicture}
{}
\caption{The two approaches to defining the tight span of a classical metric space.
}
\label{Fig:ArTightSpan}
\end{figure}

\section{Isbell completion}
In this section we define the Isbell completion of a generalized metric space as the invariant part of the Isbell adjunction.  We then give some basic examples.

\subsection{Definition of Isbell completion}
\label{Section:DefnIsbell}
%
Before defining the Isbell completion, we should first define the Isbell adjunction (or Isbell conjugation) which is described for ordinary categories by  Lawvere in~\cite{Lawvere:TakingCategoriesSeriously}.  For a generalized metric space $X$, the \definition{Isbell adjunction} is the following pair of maps between the spaces of presheaves and op-co-presheaves.
\[
\begin{tikzpicture}[>=to,->,
] 
    \matrix[matrix of math nodes,column sep={1.6cm,between origins}]
     { |(hatX)| \presheaf X & |(checkX)| \copre X^\op\\};
\draw[transform canvas={yshift=0.4ex}] (hatX) edge node [above,midway] {$L$} (checkX) ;
 \draw[transform canvas={yshift=-0.4ex}] (checkX)  edge node [below,midway] {$R$}  (hatX);
      \end{tikzpicture}\]
defined in the following way for $f\in \presheaf{X}$ and $g\in \copre{X}^\op$:
\[
L(f)(y):=\sup_{x\in X}\left(d(x,y) \mminus f(x)\right);
\qquad
R(g)(x):=\sup_{y\in X}\left(d(x,y) \mminus g(y)\right).
\]
The basic properties are summarized in the following theorem (see~\cite{Lawvere:TakingCategoriesSeriously} and~\cite{Trimble:mathoverflow}).
\begin{thm}
\label{Thm:IsbellAdjunction}
For a generalized metric space $X$, both $L$ and $R$ are short maps and commute with the Yoneda maps, so in the following diagram the two triangles commute.
\[\begin{tikzpicture}[>=to,->,
] 
    \matrix[matrix of math nodes,column sep={0.8cm,between origins}, row sep={1cm}]
     {                     &|(X)| X\\
        |(hatX)| \presheaf X && |(checkX)| \copre X^\op\\};
\draw (X) to (hatX);
\draw (X) to (checkX);
\draw[transform canvas={yshift=0.4ex}] (hatX) edge node [above,midway] {$L$} (checkX) ;
 \draw[transform canvas={yshift=-0.4ex}] (checkX)  edge node [below,midway] {$R$}  (hatX);
     \end{tikzpicture}
\]
Furthermore, $L$ and $R$ form an adjunction, so that
  \[ \dd_{\copre X^\op}(L(f),g)=\dd_{\presheaf X}(f,R(g)),\]
which means that, as in Section~\ref{Section:MetricSpaces}, the composites $RL\colon \presheaf X \to \presheaf X$ and $LR\colon \copre X^\op \to \copre X^\op$ are both idempotent: 
  \[RLRL=RL\quad\text{and}\quad LRLR=LR.\]
\end{thm}
\begin{proof}
Most of the theorem is straight forward.  We will just show that $L$ and $R$ form an adjunction.  For $f\in \presheaf X$ and $g\in \copre X^\op$ we have
\begin{align*}
  \dd_{\copre X^\op}(L(f),g)
   &=\dd_{\copre X}(g,L(f))
     =\sup_{y\in X}\bigl\{L(f)(y)\mminus g(y)\bigr\}\\
     &=\sup_{y\in X}\biggl\{\sup_{x\in X}\bigl\{\dd_X(x,y)\mminus f(x) \bigr\} \mminus g(y)\biggr\}\\
     &=\sup_{x,y\in X}\Bigl\{\dd_X(x,y)\mminus \bigl(f(x) + g(y)\bigr)\Bigr\}\\
     &=\sup_{x\in X}\biggl\{ \sup_{y\in X}\bigl\{\dd_X(x,y)\mminus g(y) \bigr\} \mminus f(x)\biggr\}\\
     &=\sup_{x\in X}\bigl\{R(g)(x)\mminus f(x)\bigr\}\\
     &=\dd_{\presheaf X}\bigl(f,R(g)\bigr),
\end{align*}
as required.
\end{proof}

We can now define the \definition{Isbell completion} $I(X)$ of a generalized metric space $X$ to be the `invariant part' of the Isbell adjunction:
\begin{align*}
  I(X)&:=\left\{(f,g)\in \presheaf X \times \copre X^\op \mid L(f) =g\text{ and } R(g)=f\right\}.
 \end{align*}
There are a couple of equivalent ways to describe this, but a few comments are in order here.  Firstly, we should make a note of the generalized metric on the Isbell completion $I(X)$.  The metric is inherited from that on $\presheaf X \times \copre X^\op$ which itself is a product metric, which from the category theory point of view means that we should use the maximum --- the categorical product in  $\Rplus$ --- of the metrics for the two factors. 
  \[\dd_{I(X)}\left( (f,g),(f',g')\right):= \max\left\{ \dd_{\presheaf X}(f,f'),\dd_{\copre X^\op}(g,g')\right\}.\]
However, things simplify somewhat as the following lemma shows.
\begin{lemma}
\label{Lemma:EasyMetric}
Suppose that $X$ is a generalized metric space.
For two points in the Isbell completion, $(f,g),(f',g')\in I(X)$, the generalized distance between them can be expressed in the two simple ways:
  \[\dd_{I(X)}\left( (f,g),(f',g')\right)= \dd_{\presheaf X}(f,f')=\dd_{\copre X^\op}(g,g').\]
\end{lemma}
\begin{proof}
Given the definition of the metric on $I(X)$ above, it suffices to show that $\dd(f,f')=\dd(g,g')$.  This is straightforward.
 \[\dd_{\presheaf X}(f,f')=\dd_{\presheaf X}(R(g),R(g'))=\dd_{\copre X^\op}(LR(g),g')=\dd_{\copre X^\op}(g,g').\]
 The middle equality uses the fact that $L$ and $R$ form an adjunction, and the other equalities use the relations $f=R(g)$, $f'=R(g')$ and $g=L(f)$.
\end{proof}
We want to think of $I(X)$ as an `extension' of $X$.  For a point of the Isbell completion, $P=(f,g)\in I(X)$, we want to think of $f$ as giving the distance of $P$ from each point of $X$ and $g$ as giving the distance of $P$ to each point of $X$.  This can be encapsulated as following theorem which is straightforward to prove.  

\begin{thm}
For $X$ a generalized metric space, the Yoneda maps give an isometry
  \[\ \Yoneda\colon X\to I(X);\qquad x\mapsto \left(\dd_X({-},x),\dd_X(x,{-})\right)\]
Then for $P=(f,g)\in I(X)$ we have
 \[
   \dd_{I(X)}\bigl(\Yoneda(x),P\bigr)= f(x);
   \qquad
   \dd_{I(X)}\left(P,\Yoneda(x)\right)= g(x).
\]
\end{thm}

We can try to understand in what sense this `extension' is minimal, in other words, try to interpret the condition on the points of the Isbell completion.  One can think of the following as saying that for every $z\in X$, every point in the Isbell completion, $P\in I(X)$, is arbitrarily close to being on a geodesic between $z$ and another point of $X$.
\begin{thm}
For every $z\in X$ and every $P\in I(X)$, if $\epsilon >0$ then there exist $x,y\in X$ such that 
 \[
 \dd(x,P)+\dd(P,z)\le d(x,z)+\epsilon
 \quad\text{and}\quad
 \dd(z,P)+\dd(P,y)\le d(z,y)+\epsilon,
 \]
i.e.~in each picture below the two paths differ by at most $\epsilon$. 
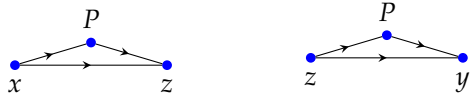
\begin{center}
\begin{tikzpicture}[
                                decoration={markings,mark=at position 0.5 with {\arrow{stealth}}}]
  \node [circle,draw=blue,fill=blue,thick, inner sep=0pt,minimum size=1mm ,label=below:$x$](x) at (0,0) {};
  \node [circle,draw=blue,fill=blue,thick, inner sep=0pt,minimum size=1mm ,label=below:$z$](z) at (2,0) {};
  \node [circle,draw=blue,fill=blue,thick, inner sep=0pt,minimum size=1mm ,label=above:$P$](P) at (1.0,0.3) {};
  \draw [postaction={decorate}](x) to (z);
  \draw [postaction={decorate}](x) to (P);
  \draw [postaction={decorate}](P) to (z);
\end{tikzpicture}
\qquad\qquad
\begin{tikzpicture}[
                                decoration={markings,mark=at position 0.5 with {\arrow{stealth}}}]
 \node [circle,draw=blue,fill=blue,thick, inner sep=0pt,minimum size=1mm ,label=below:$z$](z) at (0,0) {};
  \node [circle,draw=blue,fill=blue,thick, inner sep=0pt,minimum size=1mm ,label=below:$y$](y) at (2,0) {};
  \node [circle,draw=blue,fill=blue,thick, inner sep=0pt,minimum size=1mm ,label=above:$P$](P) at (1,0.3) {};
  \draw [postaction={decorate}](z) to (y);
  \draw [postaction={decorate}](z) to (P);
  \draw [postaction={decorate}](P) to (y);
\end{tikzpicture}
\end{center}
\end{thm}
\begin{proof}
 This is just an unpacking of the definition of a point in the Isbell completion.  So if $P=(f,g)\in I(X)$ then $L(f)=g$ and $f=R(g)$.  Considering the first condition, we have for $z\in X$ that $L(f)(z)=g(z)$ which means
 \begin{align*}
 \sup_{x\in X}\bigl\{\dd(x,z)-f(z)\bigr\}&=g(z)\\
\intertext{so}
 \sup_{x\in X}\bigl\{\dd(x,z)-\dd(x,P)\bigr\}&=\dd(P,z).
\end{align*}
Thus for all $\epsilon>0$ there exists $x\in X$ such that
\begin{align*}
 \dd(x,z)-\dd(x,P)&\ge\dd(P,z)-\epsilon\\
\intertext{whence}
  \dd(x,P)+\dd(P,z)&\le d(x,z)+\epsilon. 
 \end{align*}
The other inequality is obtained by similar considerations with $f=R(g)$.
\end{proof}
There are a couple of other ways of describing the Isbell completion which are less helpful intuitively, but are more useful in doing calculations.  We define the following generalized metric spaces:
\begin{align*}
  \Fix_{RL}(X)&:=\left\{f\in \presheaf X \mid RL(f)=f\right\}\subset \presheaf X;\\
  \Fix_{LR}(X)&:=\left\{g\in \copre X^\op \mid LR(g)=g\right\} \subset \copre X^\op.
\end{align*}
There are obvious inclusion and projection maps
\[
\begin{tikzpicture}[>=to,->,
] 
    \matrix[matrix of math nodes,column sep={1.6cm,between origins}]
     { |(FixRL)| \Fix_{RL}(X) & |(IX)| I(X)&|(FixLR)| \Fix_{LR}(X)  \\};
\draw[transform canvas={yshift=0.4ex}] (FixRL) edge node [above,midway] {} (IX) ;
\draw[transform canvas={yshift=-0.4ex}] (IX)  edge node [below,midway] {}  (FixRL);
\draw[transform canvas={yshift=0.4ex}] (FixLR) edge node [above,midway] {} (IX) ;
\draw[transform canvas={yshift=-0.4ex}] (IX)  edge node [below,midway] {}  (FixLR);
      \end{tikzpicture}.\]
%
For instance $f\in \Fix_{RL}(X)$ maps to $(f,L(f))\in I(X)$.  These maps are immediately seen to be bijections, and they are isometries by Lemma~\ref{Lemma:EasyMetric}.  So the three spaces are isometric.
  \[ \Fix_{RL}(X)\cong I(X) \cong\Fix_{LR}(X).\]
We now move on to some examples.

\subsection{Examples of Isbell completions}
We will just look at some very basic examples here, but they will demonstrate some pertinent features.
\subsubsection{Classical metric space with two points.}
\begin{figure}[tb]
\begin{center}
\begin{tikzpicture}[baseline=-2cm]
\node [circle,draw=blue,fill=blue,thick, inner sep=0pt,minimum size=1mm ,label=above:$b$](b) at (2,0) {};
\node [circle,draw=blue,fill=blue,thick, inner sep=0pt,minimum size=1mm ,label=above:$a$](a) at (0,0) {};
\draw[<->,shorten >=2pt,shorten <=2pt] (a) -- node[above]{$r$} (b);
\end{tikzpicture}
\qquad
\begin{tikzpicture}
\draw[thin,draw=black,->] (-0.1,0)  -- (2.5,0) node[below] {$f(a)$};
\draw[thin,draw=black,->] (0,-0.1) -- (0,2.5) node[left] {$f(b)$};
\draw[fill=green!80,fill opacity=0.5] (0,0) rectangle (2,2);
\node [circle,draw=blue,fill=blue,thick, inner sep=0pt,minimum size=1mm ,label=above right:$b$](b) at (2,0) {};
\node [circle,draw=blue,fill=blue,thick, inner sep=0pt,minimum size=1mm ,label=above right:$a$](a) at (0,2) {};
\node [circle,draw=blue,fill=blue,thick, inner sep=0pt,minimum size=1mm ,label=above right:$\toppoint$](T) at (2,2) {};
\node [circle,draw=blue,fill=blue,thick, inner sep=0pt,minimum size=1mm ,label=above right:$\bottompoint$](B) at (0,0) {};
\draw[draw=blue, ultra thick] (a) -- (b);
\end{tikzpicture}
\begin{tikzpicture}
\draw[thin,draw=black,->] (-0.1,0)  -- (2.5,0) node[below] {$g(a)$};
\draw[thin,draw=black,->] (0,-0.1) -- (0,2.5) node[left] {$g(b)$};
\draw[fill=green!80,fill opacity=0.5] (0,0) rectangle (2,2);
\node [circle,draw=blue,fill=blue,thick, inner sep=0pt,minimum size=1mm ,label=above right:$b$](b) at (2,0) {};
\node [circle,draw=blue,fill=blue,thick, inner sep=0pt,minimum size=1mm ,label=above right:$a$](a) at (0,2) {};
\node [circle,draw=blue,fill=blue,thick, inner sep=0pt,minimum size=1mm ,label=above right:$\bottompoint$](B) at (2,2) {};
\node [circle,draw=blue,fill=blue,thick, inner sep=0pt,minimum size=1mm ,label=above right:$\toppoint$](T) at (0,0) {};
\draw[draw=blue,ultra thick] (a) -- (b);
\end{tikzpicture}
\end{center}
\caption{The metric space $A_r$ with its Isbell completion pictured as $\Fix_{RL}(A_r)$ and $\Fix_{LR}(A_r)$, the tight span is also marked in.}
\label{Fig:AdIsbellCompletion}
\end{figure}
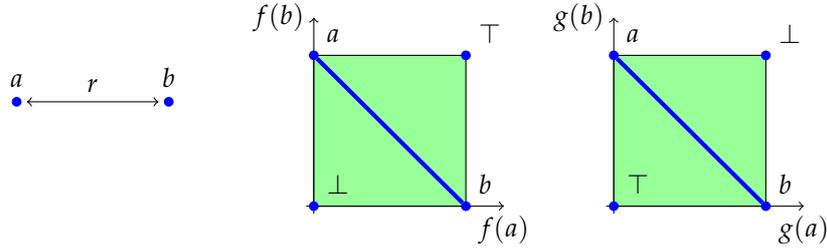
We take $A_r$ to be the symmetric metric space with two points $a$ and $b$ with a distance of $r$ between them, where $r>0$.   We use $\Fix_{RL}(A_r)$ as our definition of the Isbell completion, so that we are interested in
 \begin{multline*}
  \Fix_{RL}(A_r)=\biggl\{ f\colon\{a,b\}\xrightarrow{} [0,\infty]\mid 
       \forall x\in \{a,b\}\\
       f(x)=\max_{y\in\{a,b\}}\Bigl( \dd_{A_r}(x,y) \mminus
       \max_{z\in \{a,b\} }\bigl(\dd_{A_r}(z,y)\mminus f(z)\bigr)\Bigr)
       \biggr\}
\end{multline*}
A not very enlightening calculation allows us to deduce that this set is given by
  \[
    \Fix_{RL}(A_r)=
     \biggl\{ f\colon\{a,b\}\xrightarrow{} [0,\infty]\mid 
    0\le f(a)\le r,\ 0\le f(b)\le r   \biggr\}.
  \]
This is, of course, easily pictured as a subset of $\R^2$, as in Figure~\ref{Fig:AdIsbellCompletion}.  The generalized metric is the asymmetric $\sup$ metric, given by
  \[\dd\bigl((\alpha,\beta),(\alpha',\beta')\bigr)=\max( \alpha'-\alpha,\beta'-\beta,0).\]
There are various features in Figure~\ref{Fig:AdIsbellCompletion} which should be pointed out.  Firstly, of course, we see $A_r$ embedded isometrically in the Isbell completion.  Secondly, the tight span $T(A_r)$ is seen as the diagonal of the square, this should be compared with Figure~\ref{Fig:ArTightSpan}, see Section~\ref{Section:TightSpanInIsbell} for more on this.  Thirdly, the point $\bottompoint$ (or  ``bottom'') in the picture is a distance $0$ \emph{from} every point in the Isbell completion, $\dd_{I(X)}((f,g),\bottompoint)=0$ for all $(f,g)\in I(X)$; whereas the point $\toppoint$ (or ``top'') in the picture is a distance $0$ \emph{to} every point in the Isbell completion.  These can be thought of, in category theoretic language, as a terminal and an initial point: that the Isbell completion has these is a feature of the fact that it has all weighted limits and colimits, as will be seen later.  Also pictured in Figure~\ref{Fig:AdIsbellCompletion} is $\Fix_{LR}(A_r)$, which is a subset of $\R^2$ with the opposite asymmetric metric, this is of course also isometric with the Isbell completion, and is basically the other picture flipped over.
\subsubsection{Example of two asymmetric points}
\label{section:IsbellCompletionAsymmetricTwoPoints}
\begin{figure}[!t]
\begin{center}
\begin{tikzpicture}
\node [circle,draw=blue,fill=blue,thick, inner sep=0pt,minimum size=1mm ,label=above:$b$](b) at (3,0) {};
\node [circle,draw=blue,fill=blue,thick, inner sep=0pt,minimum size=1mm ,label=above:$a$](a) at (0,0) {};
\draw[->,shorten >=2pt,shorten <=2pt,transform canvas={yshift=0.4ex}] (a) -- node[above]{$r$} (b);
\draw[->,shorten >=2pt,shorten <=2pt,transform canvas={yshift=-0.4ex}] (b) -- node[below]{$s$} (a);
\end{tikzpicture}\\
\begin{tikzpicture}
\draw[thin,draw=black,->] (-0.1,0)  -- (4,0) node[below ] {$f(a)$};
\draw[thin,draw=black,->] (0,-0.1) -- (0,2.5) node[left] {$f(b)$};
\draw[fill=green!80,fill opacity=0.5] (0,0) rectangle (3,2);
\node [circle,draw=blue,fill=blue,thick, inner sep=0pt,minimum size=1mm ,label=above right:$b$](b) at (3,0) {};
\node [circle,draw=blue,fill=blue,thick, inner sep=0pt,minimum size=1mm ,label=above right:$a$](a) at (0,2) {};
\node[label=below:$r$] at (b) {};
\node[label=left:$s$] at (a) {};
\node[label=left:$0$] at (0,0) {};
\node[label=below:$0$] at (0,0) {};
\node [circle,draw=blue,fill=blue,thick, inner sep=0pt,minimum size=1mm ,label=above right:$\toppoint$](T) at (3,2) {};
\node [circle,draw=blue,fill=blue,thick, inner sep=0pt,minimum size=1mm ,label=above right:$\bottompoint$](B) at (0,0) {};
\end{tikzpicture}
%
%
\begin{tikzpicture}
\draw[thin,draw=black,->] (-0.1,0)  -- (3,0) node[below] {$g(a)$};
\draw[thin,draw=black,->] (0,-0.1) -- (0,3.5) node[left] {$g(b)$};
\draw[fill=green!80,fill opacity=0.5] (0,0) rectangle (2,3);
\node [circle,draw=blue,fill=blue,thick, inner sep=0pt,minimum size=1mm ,label=above right:$b$](b) at (2,0) {};
\node [circle,draw=blue,fill=blue,thick, inner sep=0pt,minimum size=1mm ,label=above right:$a$](a) at (0,3) {};
\node[label=below:$s$] at (b) {};
\node[label=left:$r$] at (a) {};
\node[label=left:$0$] at (0,0) {};
\node[label=below:$0$] at (0,0) {};
\node [circle,draw=blue,fill=blue,thick, inner sep=0pt,minimum size=1mm ,label=above right:$\bottompoint$](B) at (2,3) {};
\node [circle,draw=blue,fill=blue,thick, inner sep=0pt,minimum size=1mm ,label=above right:$\toppoint$](T) at (0,0) {};
\end{tikzpicture}
\end{center}
\caption{The metric space $N_{r,s}$ and its Isbell completion pictured as $\Fix_{RL}(N_{r,s})$ and $\Fix_{LR}(N_{r,s})$}
\label{Fig:ArsIsbellCompletion}
\end{figure}
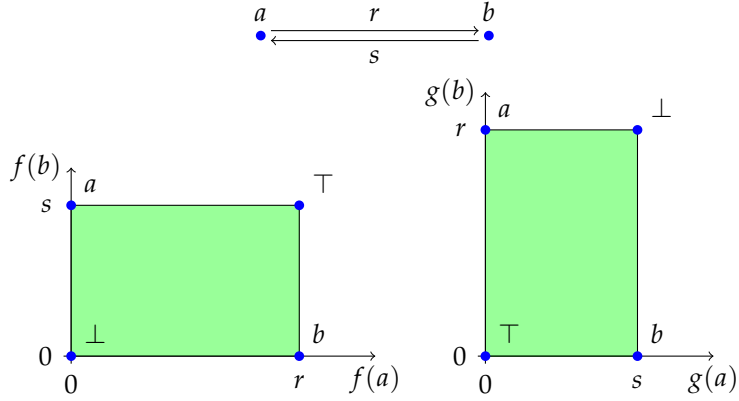
We consider the simplest asymmetric metric space $N_{r,s}$ where points $a$ and $b$ are  distances $r$ and $s$ from each other, as in Figure~\ref{Fig:ArsIsbellCompletion}.  By a calculation similar to the one above we can calculate the fixed set $\Fix_{RL}(N_{r,s})$ and hence the Isbell completion. 
 \[
 I(N_{r,s})\cong \Fix_{RL}(N_{r,s})\cong
 \Bigl\{(\alpha,\beta) \in \R^2\bigm | 0\le \alpha \le r,\ 0\le \beta \le s\Bigr\}
 \subset \R^2_{\text{asym}}
 \]
This is also pictured in  Figure~\ref{Fig:ArsIsbellCompletion}: there we see $N_{r,s}$ is clearly isometrically embedded and there are top and bottom points $\toppoint$ and $\bottompoint$.  Note that there is no tight span drawn, as the space $N_{r,s}$ is not a classical metric space and the tight span is not defined.

\subsubsection{Example of three points  with a symmetric metric}
We consider the three-point metric space $A_{r,s,t}$ with distances $r$, $s$ and $t$, where, without loss of generality we assume $r\ge s\ge t$: this is pictured in  Figure~\ref{Fig:ArstIsbellCompletion}.  The Isbell completion $I(X)$ is isometric to the fixed set $\Fix_{RL}(A_{r,s,t})$ which by definition is given by the following subset of $\R^3$ with the asymmetric generalized metric:
\begin{align*}
\bigl\{(\alpha,\beta,\gamma)\bigm | 
\alpha &=\max(t-\max(t-\alpha,r-\gamma,0),s-\max(s-\alpha,r-\beta,0),0)\\
\beta &=\max(r-\max(r-\beta,s-\alpha,0),t-\max(t-\beta,s-\gamma,0),0)\\
\gamma &=\max(s-\max(s-\gamma,t-\beta,0),r-\max(r-\gamma,t-\alpha,0),0)
\bigr\}.
\end{align*}
It transpires that this consists of four parts of planes, as pictured in  Figure~\ref{Fig:ArstIsbellCompletion}; however, this is not obvious, so we now prove it.
\begin{prop}
\label{Prop:ArstCalculation}
The Isbell completion $I(A_{r,s,t})$ for $r\ge s\ge t$ is isometric to the following subset of $\R^3_{\mathrm{asym}}$.
\begin{align*}
\bigl\{(\alpha,\beta,\gamma)\in [0,\infty]^3 &\bigm | 
 &&\alpha +r\le\beta+s =  \gamma+t\le s+t \\
  &\text{or}  &&\gamma+t\le\alpha+r=\beta+s \le r+s \\
  &\text{or}&&\beta+s \le \gamma+t = \alpha +r \le r+t\\
 &\text{or}  &&\alpha=0,\  \beta \le r-s,\  \gamma \le r-t\bigr\}
 \subset \R^3_{\mathrm{asym}}
 \end{align*}
 \end{prop}
\begin{proof}
The proof is just a calculation and not very informative.  First we make the following linear change of variable which makes things slightly less messy.
\[A:=\alpha+r,\quad B:=\beta+s, \quad C:=\gamma+t.\]
The equations to be solved then become the following:
\begin{align*}
A &=\max(\min(A,C,t+r),\min(A,B,s+r),r)\tag{a}\label{eq:A}\\ 
B &=\max(\min(B,A,r+s),\min(B,C,t+s),s)\tag{b}\label{eq:B}\\
C &=\max(\min(C,B,s+t),\min(C,A,r+t),t)\tag{c}\label{eq:C}
\end{align*}
We will show the solution consists of points $(A,B,C)$ with $r\le A$, $s\le B$, $t\le C$ such that one of the following hold:
\begin{gather}
C\le B=A\le r+s
\tag{1}\label{eq:1}\\ 
 A\le C=B\le s+t
\tag{2}\label{eq:2}\\
 B\le A=C\le r+t 
\tag{3}\label{eq:3}\\
  A=r,\ B\le r,\ C\le r.\tag{4}\label{eq:4}
 \end{gather}
These can easily be verified to be sufficient conditions by substituting them into the equations~\eqref{eq:A},~\eqref{eq:B} and~\eqref{eq:C}.  Before showing that they are necessary, note from~\eqref{eq:A} that \[r\le A \le s+r.\]
Similarly from~\eqref{eq:B} and~\eqref{eq:C}, we get
 \[s\le B\le r+s,\qquad t\le C\le r+t.\]
 We show the necessity of the conditions by splitting into cases $A=r$ and $A > r$, with the latter being split into subcases $A<B$, $A=B$ and $A> B$.
 
Consider first the case that $A=r$.   From~\eqref{eq:B} we deduce that
\[B\le r \quad\text{or}\quad(r<B\le C\ \text{and}\ B\le s+t).\]
From~\eqref{eq:C} we deduce that
\[C\le r \quad\text{or}\quad(r<C\le B\ \text{and}\ C\le s+t).\]
From these we can deduce that
\[(B\le r \ \text{and}\  C\le r)\quad\text{or}\quad r=A<C=B\le s+t .\]
which are covered by~\eqref{eq:2} and~\eqref{eq:4}.

Consider now the case $A>r$.  We will look at the subcases $A<B$, $A=B$ and $A> B$.
Suppose that $A<B$.   Then also $s\le r<A<B$ so~\eqref{eq:B} gives $B\le C$ and $B\le s+t$.  From~\eqref{eq:C} we deduce that either $C\le B$ or $(B<C\text{ and }C=t)$, but the latter can't hold as we know $t\le s < B$ thus the former is true, and so we deduce $B=C$.  Putting this together, we conclude
 \[r< A<C =B\le s+t,\]
which is covered by~\eqref{eq:2}.

Suppose $A=B$.  From~\eqref{eq:C} then either $C\le A$ or $A< C=t$ but $t\le r\le A$ so the latter can not hold and so we deduce
 \[t \le C \le A = B \le r+s,\]
which is covered by~\eqref{eq:1}.

Suppose finally that $B<A$.  From~\eqref{eq:A} we see that $A\le C$ and $A\le t+r$.  As we know that $t\le s\le B<A$ we deduce that $t\le B<C$.  So from~\eqref{eq:C} we see that $C\le A$ and $C\le r+t$.  Putting this all together we get
 \[s\le B<A=C\le r+t,\]
which is covered by~\eqref{eq:3}.

Thus equations~\eqref{eq:1}--\eqref{eq:4} give necessary and sufficient conditions for the solutions of the original equations.
\end{proof}

\section{The Isbell completion and the tight span}
In this section we show that the tight span of a classical metric space is contained in the Isbell completion.  We then prove that the directed tight span of Hirai and Koichi coincides with the Isbell completion.
\subsection{The classical tight span}
\label{Section:TightSpanInIsbell}
Here we show that the Isbell's tight span of a classical metric space is contained in the Isbell completion of the metic space.
\begin{thm}
\label{Thm:TightSpanInIsbell}
If $\X$ is a classical metric space then the tight span $T(\X)$ is (isometric to) the maximal classical metric space inside the generalized metric space $I(\X)$ which contains the image of $\X$. 
\end{thm}
\begin{proof}
In Section~\ref{Section:DefnOfTightSpan} the tight span $T(\X)\subset \presheaf{\X}$ is characterized by
  \[T(\X)=\bigl\{f\in \presheaf{\X}\bigm | f(x)=\sup_{y \in \X} (\dd(x,y)\mminus f(y)) \quad \text{for all }x\in \X\bigr\} .\]
As $\X$ is symmetric, presheaves and copresheaves are the same thing, so there is a well-defined function $T(\X)\hookrightarrow I(\X)$ given by $f\mapsto (f,f)$.  Indeed it can be seen that $T(\X)$ is actually a subset of $\Fix_{RL}(X)\in \presheaf{\X}$ and so inherits the same generalized metric, meaning that the inclusion is an isometry. 

To see that it is the maximal classical metric subspace containing $\X$ observe than any point $(f,g)\in I(\X)$ inside a classical metric subspace containing $\X$ must have the same distance both to and from any given point of $\X$, in other words we must have $f=g$, thus $(f,g)\in T(\X)$ and so $T(\X)$ is maximal. 
\end{proof}
This is illustrated in Figures~\ref{Fig:AdIsbellCompletion} and~\ref{Fig:ArstIsbellCompletion}.
\subsection{The directed tight span of Hirai and Koichi}
For a finite generalized metric space with no infinite distances, Hirai and Koichi~\cite{HiraiKoichi:DirectedDistances} defined a `tight span' which they denoted $T_\mu$, where $\mu$ was the generalized metric.  For a generalized metric space $X$ we will use the notation $\HK(X)$ and refer to it as the Hirai-Koichi tight span.  Their definition is analogous to what, for a classical metric space $\X$, we called the injective envelope $\EE(\X)$.  The definition can be easily generalized to all generalized metric spaces as follows.

Let $P(X)$ be the set of pairs $(f,g)$ of functions $f,g\colon X \to [0,\infty]$ such that
  \[f(x)+g(y)\ge \dd(x,y)\quad\text{for all }x,y \in X.\]
We will call such pairs \definition{triangular}, as the above inequality is supposed to be suggestive of the triangle inequality:  $f(x)$ being like the distance from $x$ to the `point' $(f,g)$ and $g(y)$ being like the distance from that `point' to $y$.  A triangular pair $(f,g)\in P(X)$ is said to be \definition{minimal} if for any triangular pair $(f',g')\in P(X)$ such that for all $x,y\in X$ we have $f'(x)\le f(x)$ and $g'(y)\le g(y)$ then $(f',g')=(f,g)$.  The \definition{Hirai-Koichi tight span} $\HK(X)$ is defined to be the set of minimal triangular pairs in $P(X)$, it is equipped with the following metric:
  \[
    \dd_{\HK(X)}\left( (f,g),(f',g')\right) := 
    \sup_{x\in X} \left\{\max\left( f'(x)\mminus f(x),g'(x)\mminus g(x)\right)\right\}.
  \]
We can now show that this coincides with the Isbell completion.
\begin{thm}
\label{Thm:HiraiKoichi}
 For a generalized metric space $X$, the Isbell completion and Hirai-Koichi tight span coincide, conisdered as subspaces of the same generalized metric space:
  \[I(X)=\HK(X)\subset \left( [0,\infty]^X \times [0,\infty]^X\right)_{\mathrm{asym}}.\]
\end{thm}
\begin{proof}
  As both the Hirai-Koichi tight span and the Isbell completion are both submetric spaces of the same generalized metric spaces it suffices to prove that $\HK(X)\subset I(X)$ and that $I(X)\subset \HK(X)$.
  
  To show that $\HK(X)\subset I(X)$ we need to show that for any minimal triangular pair $(f,g)\in P(X)$ we have that $f$ is a presheaf, $g$ is a co-presheaf, $L(f)=g$ and $f=R(g)$.
 To see that $f$ is a presheaf, fix $x$ and $y$ and define a function $f_{xy}$ by 
  \[
    f_{xy}(w):=
    \begin{cases}
      f(y)+\dd(x,y)& \text{for }w=x\\
      f(w) &\text{otherwise}
      \end{cases}
  \]
then the pair $(f_{xy}, g)$ is triangular  and  so by the minimality of $(f,g)$ we must have $f_{xy}(x)\ge f(x)$.  From this we obtain
  \[\dd(x,y)\ge f(x)\mminus f(y)\]
and so $f$ is a presheaf.   The proof that $g$ is a co-presheaf is analogous .

Now to show that $g=L(f)$, fix $y\in X$.  As $(f,g)$ is triangular,
 \[g(y)\ge \dd(x,y)\mminus f(x)\quad \text{for all }x\in X.\]
So $g(y)\ge L(f)(y)$ for all $y\in X$.  But it is easy to check that $(f,L(f))$ is triangular.  Thus by the minimality of $(f,g)$ we must have $g=L(f)$, as required.
Similarly, $f=R(g)$.  Thus $(f,g)\in I(X)$ and so $\HK(X)\subset I(X)$.

We now have to show $I(X)\subset \HK(X)$.  Suppose $(f,g)\in I(X)$, then $g=L(f)$ and $f=R(g)$, we need to show that $(f,g)$ is triangular and minimal.  Fix $x,y\in X$.  As $f=R(g)$ we have
  \[f(x)=\sup_{z\in X}\left\{\dd(x,y)\mminus g(z)\right\}\ge \dd(x,y)\mminus g(y),\]
so $f(x)+g(y)\ge \dd(x,y)$ and $(f,g)$ is triangular. 

To show minimality, suppose that $(f'',g'')\in P(X)$ is a triangular pair which satisfies 
  \[f''(x)\le f(x), \ g''(y)\le g(y)\quad \text{for all } x,y \in X.\]
Fix $x\in X$.  By the least upper bound property of $L(g)$ we know that for all $\epsilon >0$ there exists a $y\in X$ such that $f(x)\le \dd(x,y)\mminus g(y)+\epsilon$.  Thus
  \[f''(x)\le f(x)<\dd(x,y)\mminus g(y)+\epsilon\le f''(x)+g''(y)-g(y)+\epsilon\le f''(x)+\epsilon.\]
Thus $f(x)=f''(x)$, and so $f=f''$.  Similarly,  $g=g''$  and so $(f,g)$ is minimal, hence $(f,g)\in \HK(X)$, whence we deduce $I(X)\subset \HK(X)$, and we can conclude $I(X)=\HK(X)$ as required.
\end{proof}

\section{Weighted colimits and semi-tropical modules}
In this section we see the definition and some examples of the notion of weighted colimit in the context of generalized metric spaces.  We then see how the idea of having all colimits is related to being a module over the semi-tropical semiring.  There is a completely dual story involving the notion of weighted limit.  We will need the dual story but will summarize those results later.
\subsection{Definition}
We will look at the definition of weighted colimit and associated concepts and at some examples.

The concept of weighted (or indexed) colimit is natural in the context of enriched categories  and generalizes the concept of ordinary colimit in ordinary category theory.  We will usually just refer to \emph{colimits}, dropping the prefix \emph{weighted}.

In order to define a weighted colimit in a generalized metric space $X$ we need some ingredients.  A \definition{shape} $D$ is a generalized metric space; a \definition{diagram} of shape $D$ is a short map $J\colon D\to X$; and a \definition{weighting} is a presheaf $W\colon D^{\op}\to [0,\infty]$.  A \definition{colimit} of the diagram $J$ with weighting $W$ is an object $c$ of $X$ which satisfies 
  \[
   \dd_X( c, x ) = \sup_{d\in D} \bigl\{\dd_X(J(d),x)\mminus W(d)\bigr\}
   \qquad\text{for all }x\in X.
   \]
This is not a terribly enlightening definition but hopefully the meaning will become clearer after some examples below, we perhaps want to think of it as a weighted sum,
  \[\text{``}c=\sum_{d\in D}W(d)\cdot J(d)\text{''},\]
or, to use symbols more commonly used in tropical mathematics, 
  \[\text{``}c=\bigoplus_{d\in D}W(d)\odot J(d)\text{''}.\]
This idea will be made precise later.
There is a slight notation clash here in that $\oplus$ is often used in category theory for things which are both products and coproducts, but we won't mean that sense here.

For a given $W$ and $J$ a colimit might or might not exist, and if it does exist then in general there might be other isomorphic colimits.  However, if $X$ is skeletal then colimits, when they exist, are uniquely defined.  We write $\colim{W}{J}$ for `the' colimit of $W$ and $J$ if it exists, so if $c$ is a colimit we can write $c\isomorphic  \colim{W}{J}$.  
   
A generalized metric space which has colimits for all diagrams and weightings is said to be \definition{cocomplete}.  If it has colimits for all diagrams and weightings on shapes which are finite then it is said to be \definition{finitely cocomplete}.

A short map $F\colon X\to Y$ is called \definition{cocontinuous} if it preserves colimits in the sense that for any $J\colon D\to X$ and $W\colon D^\op\to[0,\infty]$ for which the colimit $\colim{W}{J}$ exists we have $F(\colim{W}{J})\isomorphic \colim{W}{F\circ J}$.  It is worth recording here for later use that left adjonts are cocontinuous.
\begin{lemma}
\label{lem:LeftAdjCocont}
If $F\colon X\to Y$ is a short map is short map between generalized metric space with a right adjoint then $F$ is cocontinuous.
\end{lemma}
\begin{proof}
Let $G\colon Y\to X$ be the right adjoint.  Suppose that $J\colon D\to X$ is a diagram and $W\colon D^\op\to [0,\infty]$ is a weighting such that the colimit $\colim{W}{J}$ exists.  Then we use the adjointness twice to see
\begin{align*}
  \dd_Y(F(\colim{W}{J}),y)&=
  \dd_X(\colim{W}{J},G(y))\\
  &=
  \dd_{\presheaf D}\Bigl(W({-}),\dd_X\bigl(J({-}),G(y)\bigr)\Bigr)\\
  &=  \dd_{\presheaf D}\Bigl(W({-}),\dd_X(F\circ J({-}),y)\Bigr),
\end{align*}
thus $F(\colim{W}{J})\isomorphic \colim{W}{F\circ J}$.  So $F$ is cocontinuous as required.
\end{proof}
   
Note that the definition of the colimit does not say anything about the distance \emph{to} the colimit; we will come back to this later in Section~\ref{Section:FormalColimits}.

\subsection{Examples}
We can now have a look at some examples.

\subsubsection{Initial points}
If the shape $D$ is the empty metric space $\emptyset$ then we use the unique maps $ \emptyset\to X$ and $ \emptyset\to [0,\infty]$, and as the supremum of an empty subset of $[0,\infty]$ is $0$ we get that a colimit, denoted by $\toppoint$ satisfies
 \[\dd(\toppoint,x)=0\qquad\text{for all }x\in X.\]
So a colimit $\toppoint$ of the empty diagram is to be thought of as some sort of `initial' point, being a distance $0$ \emph{to} every other point.  Classical metric spaces with more than one point can not have an initial point, but as we have seen, the Isbell completion of a metric space can have an initial point.  The asymmetric extended non-negative reals $\Rplus$ has $\infty$ as an initial point.

\subsubsection{Fat out points}
 If the shape $D$  is a singleton metric space $\{\ast\}$ with the diagram $J$ mapping to the point $x_0\in X$ and the weighting $W$ taking the value $r\in [0,\infty]$, then the colimit $O(x_0,r)$ satisfies
  \[\dd_X( O(x_0,r), x ) =\dd_X(x_0,x)\mminus r,\]
so can be thought of as a ``fat out point'' at $x_0$ of radius $r$.  See Figure~\ref{Fig:FatOutPoint}.  
I'd rather think of this as a fat point than a ball, as a ball has many points in it.
\begin{figure}[ht]
\begin{center}
\begin{tikzpicture}
  \draw [draw=red,fill=red!20] (0,0) circle (1);
  \draw [->] (0,0) to node[auto] {$r$} (30:1);
  \node [circle,draw=blue,fill=blue,thick, inner sep=0pt,minimum size=1mm ,label=below:$x_0$](x0) at (0,0) {};
  \node [circle,draw=blue,fill=blue,thick, inner sep=0pt,minimum size=1mm ,label=below:$x$](x) at (-10:5) {};
  \draw [->](-10:1) to node [auto,pos=0.15,sloped] {$\dd(O(x_0,r),x)$} (x);
\end{tikzpicture}
\end{center}
\caption{A fat out point.}
\label{Fig:FatOutPoint}
\end{figure}
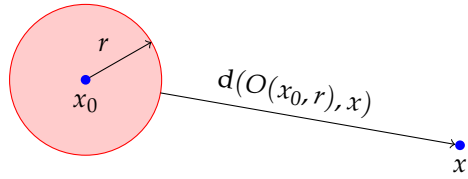

A classical metric space will typically not have such colimits.  The asymmetric metric space $[0,\infty]$ does have such fat out points: in that case $O(x_0,r)=x_0+r$.

\subsubsection{Coproducts}
If the shape $D$ consists of two points infinitely far apart, and we take the weighting $W$ to be the zero weighting, then writing $x_1,x_2\in X$ for the images of the two points under $J$ and $x_1\sqcup x_2$ for the colimit, we have
  \[\dd(x_1\sqcup x_2,x)= \max\bigl(\dd(x_1,x),\dd(x_2,x)\bigr)\qquad \text{for all }x\in X.\]
So $x_1\sqcup x_2$ should be thought of as a coproduct.  Again, in general, classical metric spaces will not have coproducts, however $[0,\infty]$ does, with $x_1\sqcup x_2=\min(x_1,x_2)$.

\subsubsection{Limits of sequences}
If the shape $D$ is the set of natural numbers $\N$ with infinite distance between each pair of points, then a diagram $\N\to X$ is the same thing as a sequence in~$X$.   Rutten~\cite[Propostion~1.1]{Rutten:WeightedColimitsMetricSpaces} showed how Cauchy sequences can be examined using weighted limits and colimits.   

\subsubsection{The asymmetric space $\Rplus$ is cocomplete}

The generalized metric space $\Rplus$ with its standard asymmetric metric has all colimits.  Given a diagram $J\colon D\to \Rplus$ and a weighting $W\colon D^{\op}\to \Rplus$ the colimit is given by
  \[\colim{J}{F}=\inf_{d\in D}(W(d)+J(d)).\]
I will leave the reader to draw the appropriate pictures of fat points, and deduce that this is immediately clear.  On the other hand, the symmetric metric structure of $[0,\infty]$ does not have such arbitrary colimits.

\subsection{Characterizing cocompleteness}
If we wish to show that a generalized metric space has all colimits then it suffices to consider shapes in which all of the the distances are infinite, we call these \definition{discrete} shapes.  If $D$ is a generalized metric space, then there is a discrete metric space $D^\delta$ called the \definition{discretization} of $D$ which has the same points as $D$, but has all non-zero distances becoming infinite.  There is an identity-on-points short map $\delta\colon D^\delta \to D$ and we can pull back any $D$-shaped diagram to a $D^\delta$-shaped diagram. 
Now using this we see that 
every colimit can be written as a colimit of a discrete diagram.
\begin{thm}
If $J\colon D \to X$ is a diagram then by pulling back we get a diagram $\delta^* J\colon D^\delta\to X$ and a discrete weighting $\delta^*W\colon D^\delta\to [0,\infty]$.  If $\colim{\delta^*W}{\delta^*J}$ exists then
  \[\colim{\delta^*W}{\delta^*J}\isomorphic \colim {W}{J}.\]
\end{thm}
\begin{proof}
If the colimit exists then we have
 \begin{align*}
   \dd_X(\colim{\delta^*W}{\delta^*J},x)
   &=
   \dd_{\presheaf{D^\delta}}\Bigl(\delta^*W\left({-}\right),\dd_X\bigl(\delta^*J\left({-}\right),x\bigr)\Bigr)\\
   &=\sup_{d\in D^\delta}\left\{\dd_{[0,\infty]}
        \Bigl(W\circ\delta(d),\dd_X\bigl(J\circ \delta(d),x\bigr)\Bigr)\right\}\\   
   &=\sup_{d\in D}\biggl\{\dd_{[0,\infty]}
         \Bigl(W({d}),\dd_X\bigl(J({d}),x\bigr)\Bigr)\biggr\}\\
   &=
   \dd_{\presheaf{D}}\Bigl(W({-}),\dd_X\bigl(J({-}),x\bigr)\Bigr) .     
  \end{align*}
Thus $\colim{\delta^*W}{\delta^*J}\isomorphic \colim {W}{J}$ as required.
\end{proof}
We immediately get the following corollary.
\begin{cor}
If $X$ has colimits for all weightings on \emph{discrete} shapes then $X$ has all colimits, i.e.~$X$ is cocomplete.

Similarly, If $X$ has colimits for all weightings on \emph{finite} discrete shapes then $X$ has all finite colimits, i.e.~$X$ is finitely cocomplete.
\end{cor}
We can now characterize cocomplete generalized metric spaces as follows.
\begin{lemma}
\label{Lem:FatOutAndCoproducts}
If a generalized metric space $X$ has an binary products and all fat out points then it has all finite colimits.
\end{lemma}
\begin{proof} By the above it suffices to prove that any discrete diagram $J$ and weighting $W$ has a colimit.  We proceed by induction on the number of objects in the shape $D$.  If $D=\emptyset$ then the colimit is the initial point which is given by the fat out point $O(x, \infty)$ for any $x\in X$.  If $D$ is not empty then we can write $D=D'\cup \{d\}$ and it is straightforward to show that $\colim{W}{J}=\colim{{W|_{D'}}}{J|_{D'}} \sqcup O(J(d),W(d))$.  The colimit on the right hand side exists by the inductive hypothesis.
\end{proof}

\subsection{Semi-tropical modules and colimits}
\label{Section:SemiTropical}
We can relate the notion of cocompletion to that of an action by the semiring $[0,\infty]$.
 
A \definition{semiring} (sometimes called a \definition{rig}) $(R,\oplus,\odot)$ is a ring without additive inverses, so there is A commutative addition $\oplus$ and a multiplication $\odot$ which have units $\zerounit$ and $\oneunit$ respectively, such that $\odot$ distributes over $\oplus$.   
A \definition{unital semiring morphism} will mean a function between semirings which preserves addition, multiplication and both units, $\zerounit$ and $\oneunit$.

We can see some examples.
\begin{itemize}
\item One standard example is $(\N_{\ge 0},+,\times)$ the non-negative integers with the usual addition and multiplication.  
\item Another example is $((-\infty,\infty],\min,+)$ which is the real numbers extended up to positive infinity, with the 	addition in the semiring being minimum and multiplication in the semiring being usual addition, the additive unit is $\infty$ and the multiplicative unit is $0$.  This is one of the variants of the \definition{tropical} semiring, sometimes $\max$ is used and the negative infinity is adjoined.   
\item The semiring of interest in this paper is $([0,\infty],\min,{+})$ and we will refer to it as the \definition{semi-tropical semiring} as it is half of the tropical semiring.  \item 
Another family takes the following form: for $(M,\oplus, \zerounit)$ a commutative monoid, the self monoid maps form a semiring $(\End(M),\oplus,\circ)$.
\end{itemize}

A \definition{module} (sometimes called a \definition{semi-module}) over a semiring $R$ is a commutative monoid $(M,\oplus,\zerounit)$ with an action $\odot\colon R\times M\to M$ satisfying the usual axioms for a module, which in this case can be most compactly summarized by saying that the function $r\mapsto r\odot{-}$ gives a unital semiring morphism $R\to \End(M)$.  An \definition{module morphism} between two $R$-modules is a map of monoids which preserves the $R$-action.  Clearly, any semiring can be thought of as a module over itself.

We are interested in modules over the semi-tropical  semiring $\bigl( [0,\infty],\min,+\bigr)$ and  will call such a module a \definition{semi-tropical module}.  The basic example is, unsurprisingly, the self-action on $\bigl( [0,\infty],\min\bigr)$ via addition.    A less obvious example is the monoid $\bigl( [0,\infty],\max\bigr)$ which becomes a semi-tropical module when we allow
 $\bigl( [0,\infty],\min,+\bigr)$ to act by truncated difference, so that $r\odot m:= m\mminus r$.  It is a straightforward calculation to show that this gives an action.

We are interested in semi-tropical module structures on generalized metric spaces and will require the action to interact with the generalized metric.  The way we wish to do so is encapsulated in the following definition.  A \definition{co-metric semi-tropical module} is a generalized metric space $(M,\dd)$ with a semi-tropical module structure $((M,\oplus),\odot)$ such that for every point $x\in M$ the distance to $x$ gives a semi-tropical module morphism
  \[\dd({-},x)\colon (M,\oplus)\to ([0,\infty],\max).\]
In other words,
  \[\dd(r\odot m\oplus s\odot n,x)=\max(\dd(m,x)\mminus r, \dd(n,x)\mminus s).\]
One example of such a thing would be the semi-tropical module $([0,\infty],\min)$ equipped with the standard asymmetric metric.  This example illustrates the following theorem, c.f.~\cite[Theorem~6.6.14]{Borceux:Handbook2}.
\begin{thm}
\label{Thm:CocompleteCometricModule}
A skeletal generalized metric space is finitely cocomplete if and only if can be given a co-metric semi-tropical module structure.
\end{thm}
\begin{proof}
Suppose that $X$ is a skeletal generalized metric space and $((X,\oplus),\odot)$ is a co-metric semi-tropical module structure.  Then $X$ has all fat out points and all coproducts: for all $r\in[0,\infty]$ and $x\in X$ we have that $r\odot x$ is a fat out point $O(x,r)$ as $\dd(r\odot x,y)=\dd(x,y)\mminus r$; and for all $x,x'\in X$ we have $x\oplus x'$ is the coproduct of $x$ and $x'$ because $\dd(x\oplus x',y)=\max(\dd(x,y),\dd(x',y))$.  From Lemma~\ref{Lem:FatOutAndCoproducts} we deduce that $X$ has all finite colimits, as required.

Conversely, suppose that $X$ has all finite colimits.  As $X$ is skeletal, we know that there is a unique initial point $\toppoint$, and we know that for each pair $x,x'\in X$ there is a \emph{unique} coproduct $x\sqcup x'$.  This makes $X$ into a monoid with $\toppoint$ as the unit.  The associativity of $\sqcup$ is a standard straightforward calculation, as is the unit axiom for $\toppoint$:
  \[\dd(x\sqcup \toppoint, y) =\max(\dd(x,y),\dd(\toppoint,y))=\max(\dd(x,y),0)=\dd(x,y).\]
 Similarly, by cocompleteness and skeletalness, given $r\in [0,\infty]$ and $x\in X$ there is a unique fat out point $O(x,r)$ and this fact is used to define the $[0,\infty]$-action: $r\odot x:=O(x,r)$. 
 
The fact that the gives a module structure is easy to verify.  Similarly, the co-metric structure follows easily.
\end{proof}
This means that we have to some extent achieved the goal of writing a colimit as some sort of sum.  If $X$ is a skeletal, cocomplete generalized metric space then for a finite shape $D$ with weighting $W$ and diagram $J$, we can rigorously write 
  \[
    \colim{W}{J}=\bigoplus_{d\in D}W(d)\odot J(d).
  \]

\section{Spaces of presheaves as universal cocompletions}
In this section we look at the space of presheaves as the universal cocompletion.  This is analogous to the the set of formal sums of elements of a set being the universal linearization of the set.  We start by considering the useful idea of pushing forward a presheaf.  We then look at presheaves as formal colimits before seeing that the presheaves form a universal cocompletion.  This is tersely covered in Kelly's book~\cite{Kelly:EnrichedCategoryTheory}.

\subsection{Pushing weights forward}
Given a short map $G\colon Y\to Z$ between generalized metric spaces we get short maps between the spaces of co-presheaves (or functionals) $\copre{Y}$ and $\copre{Z}$.  On the one hand we get the pull-back $G^*\colon \copre{Z}\to \copre{Y}$.  On the other hand there are push-forward maps $G_{*},G_{!}\colon \copre{Y}\to \copre{Z}$, sometimes called Kan extensions.

Given a short map $G\colon Y\to Z$ we can pull back any functional (or co-presheaf) $V\colon Z\to [0,\infty]$ to a functional  $G^*V\colon Y\to [0,\infty]$, defined by $G^*V(y)=V(G(y))$.  This process gives rise to a short map $G^*\colon \copre Z\to \copre Y$ between spaces of functionals.  There are also two maps in the other direction $G_!,G_*\colon \copre Y\to \copre Z$ which can be called left and right push-forwards or left and right Kan extensions.  These are defined on a function $W\colon Y\to [0,\infty]$ as follows.
\begin{align*}
G_!W(z)&:= \inf_{y\in Y}\left(W(y)+\dd_Z(G(y),z)\right)\\
G_*W(z)&:= \sup_{y\in Y}\left(W(y)\mminus\dd_Z(z,G(y))\right)
\end{align*}
The maps $G_!$ and $G_*$ are respectively left and right adjoint to the pull-back operation $G^*$, in other words we have
 \[\dd_{\copre Z}(G_!W,V)=\dd_{\copre Y}(W,G^*V)
 \quad\text{and}\quad
 \dd_{\copre Y}(G^*V,W)=\dd_{\copre Z}(V,G_*W).
\]
As observed by Lawvere in~\cite{Lawvere:MetricSpaces}, in the case that $G$ is an isometric embedding $Y\subset Z$, then the pull-back $G^{*}V$ is just the restriction of $V$ to $Y$, whereas the push-forwards $G_!W$ and $G_*W$ are extensions of $W$ to the whole of $Z$.  They are respectively the smallest and largest such extensions.  This is easy to see as if $V$ is an extension of $W$ then $W=G^*V$, but
 \[0=\dd(G^*V,G^*V)=\dd(G_!G^*V,V)=\dd(G_!W,V)\]
so $G_!W(z)\ge V(z)$ for all $z\in Z$, which means that $G_!W$ is the biggest extension of $W$.  The result for $G_*W$ is similar.

In the linear tropical notation we would write
  \[G_{!}W(z)=\bigoplus_{y\in Y}  \Bigl( W(y)\odot \dd_Z\bigl(G(y),z\bigr)\Bigr).\]
Here we can think of $\dd_Z$ as being akin to a Dirac delta-function: $\dd_Z(z',z)$ is the unit element if it is no distance from $z'$ to $z$ and it is the zero element if is infinite distance from $z'$ to $z$.  From that, we see that $G_{!}W(z)$ is akin to
  \[\bigoplus_{y\in G^{-1}(z)} W(y),\]
which is what you would expect for the push-forward if $G$ were a map between sets.
\begin{thm}\label{Thm:PushingForwardWeightings}
If $J\colon D\to X$ is a diagram in $X$ and $W\colon D^{\op}\to \Rplus$ is a weighting, then suppose that we have a map $G\colon X\to Y$.  We can push forward the weighting along $J$, so that we have a weighting $J_! W\colon X^{\op}\to \Rplus$ and a diagram $G\colon X\to Y$ in $Y$.  We then find that this gives us the same colimit as if we had composed the original diagram with $G$:
  \[\colim{J_!W}{G}=\colim{W}{J^* G}.\]
\end{thm}  
\begin{proof}This is a straightforward calculation.
\begin{align*}
\dd_Y(\colim{J_!W}{G},y)&=
\dd_{\presheaf X}\bigl(J_!W({-}),d_Y(G({-}),y)\bigr)\\
&=
\dd_{\presheaf D}\bigl(W({-}),J^*d_Y(G({-}),y)\bigr)\\
&=
\dd_{\presheaf D}\bigl(W({-}),d_Y(G(J({-})),y)\bigr)\\
&=
\dd_{\presheaf D}\bigl(W({-}),d_Y(J^*G({-}),y)\bigr).
\end{align*}
Thus, by the definition of the colimit, $\colim{J_!W}{G}\isomorphic \colim{W}{J^*G}$.  
\end{proof}

In the linear tropical notation the conclusion of the theorem would read as
  \[\bigoplus_{x\in X}  \bigl( J_{!}W(x)\odot G(x)\bigr)
=\bigoplus_{d\in D}  \bigl( W(d)\odot G\bigl(J(d)\bigr)\bigr)\]
which fits in with the intuition of $ J_{!}W(x)$ being akin to  
$\bigoplus_{d\in J^{-1}(x)} W(d)$ as mentioned above.

\subsection{Formal colimits and the space of presheaves}
\label{Section:FormalColimits}
We can show that all colimits, and indeed all potential colimits, in a generalized metric space $X$ are parametrized by the space of presheaves $\presheaf{X}$.  If we have a diagram $J\colon D\to X$ and a weighting $W\colon D^{\op}\to \Rplus$, then, using Theorem~\ref{Thm:PushingForwardWeightings} above, we can push forward the weighting along $J$ and use the identity map $\id\colon X\to X$ as a diagram and see that
 \[\colim{J_!W}{\id}=\colim{W}{J}.\]
This means that any weighted colimit on $X$ is expressible as a colimit on the identity diagram for some weighting in $\presheaf{X}$.  In this way we can think of the space of presheaves $\presheaf X$ as the space of `formal' colimits in $X$.  We say `formal' as the colimits might not exist in $X$, or else two presheaves in $\presheaf{X}$ might give rise to the same colimit in $X$.

There is a potential point for confusion here in that weighted colimits naturally are defined in terms of a \emph{co}presheaf, namely the distance \emph{from} the colimit: 
  \[\dd(\colim{W}{J}, {-}):=\dd_{\presheaf{D}}\bigl(W(?),\dd_{X}(J(?),{-})\bigr).\]
For a formal colimit $f\in \presheaf{X}$ we find
  \[\dd(\colim{f}{\id}, {-}):=\dd_{\presheaf{X}}\bigl(f(?),\dd_{X}(?,{-})\bigr)=L(f)({-}).\]
So the associated co-presheaf is precisely the Isbell conjugate of the presheaf.  

For example, the fat out point at $j\in X$ of radius $w\in [0,\infty]$ is associated to the presheaf $f_{j,w}\in \presheaf{X}$ given by
  \[f_{j,w}(x)=\dd(x,j) + w,\]
and the corresponding co-presheaf on $X$ is
  \[L(f_{j,w})(x)=\dd(j,x)\mminus w.\]
It would seem that $f$ is what the distance \emph{to} the colimit `should' be if it existed in $X$.  We can now see that $\presheaf X$ is itself cocomplete and can be viewed as a universal cocompletion of $X$.

\subsection{Presheaf spaces as universal cocompletions}
The presheaf space $\presheaf {X}$ which we are thinking of as the formal colimits in $X$, can be thought of as being analogous to the set $\tilde A$ of formal linear combinations of elements of a set $A$, say, for concreteness, over a field $k$.  In that case, because $k$ has a linear structure on it, $\tilde A$ can be equipped with a linear structure, which is defined pointwise.  In the case of $\presheaf X$, it has all colimits and these can be calculated pointwise using the fact that the `scalars' $\Rplus$ have all colimits.  We will expand on this analogy below, but note that there is a subtle distinction between the two cases: having a linear structure is, as the name suggests, extra \emph{structure} on a set, whereas having all colimits is a \emph{property} of a generalized metric space.

If $A$ is a set, then the set of formal linear combinations $\tilde A$ can be identified with the set of functions $\Set(A,k)$.  The set $\tilde A$ can be equipped with a linear structure pointwise, using the linear structure in $k$:
  \[(r\cdot f + s\cdot g)(i):=r\cdot f(i)+s\cdot g(i)\qquad \forall r,s\in k;\ f,g\in \tilde A;\ i\in A.
  \tag{$\ddagger$}\label{ddagger}\]
There is an obvious embedding of sets $\Yoneda\colon A\to \tilde{A}$ where an element $i$ gets sent to the delta function at $i$.  
Furthermore, $\tilde{A}$ has the following universal property with respect to linear maps.  If $F\colon A\to B$ is a function between sets, and $B$ is equipped with a linear structure then there is a unique linear map $\tilde{F}\colon \tilde{A}\to B$ such that $F$ factors through it: $F=\tilde{F}\circ \Yoneda$. 
This familiar example might be useful to bear in mind when thinking about the following theorem, which is the generalized metric space version of an enriched category theorem~\cite{Kelly:EnrichedCategoryTheory}.

%
%
%

\begin{thm}
\label{Thm:PresheavesCocomplete}
If $X$ is a generalized metric space then the presheaf space $\presheaf{X}$ is cocomplete and the colimits can be calculated pointwise as follows.  For a weighting $W\colon D^\op\to [0,\infty]$ and a diagram $J\colon D\to \presheaf{X}$, for each $x\in X$ there is a diagram $J_x \colon D\to \Rplus$ defined by $J_x(d):=J(d)(x)$.  The colimit of $W$ and $J$ is then given by
  \[\bigl(\colim{W}{J}\bigr)(x)=\colim{W}{J_x},\]
 where the right-hand side is a colimit in $\Rplus$.
 
Furthermore, if $F\colon X\to Y$ is a short map to a cocomplete generalized metric space $Y$ then the map  is a unique-up-to-isomorphism cocontinuous map $\presheaf{F}\colon \presheaf{X}\to Y$ such that the $F$ factors through $\presheaf{F}$ and the Yoneda embedding:
 \[
  \begin{tikzpicture}[text height=1.5ex,text depth=.25ex, arrowlabel/.style={font=\footnotesize}]
    \node (C) {$X$};
    \node (Chat) [right of = C] {$\presheaf{X}$};
    \node (E) [below of = Chat] {$Y$};
    \draw [->] (C) to node [auto,arrowlabel] {$\Yoneda$} (Chat);
    \draw [->] (Chat) to  node [auto,arrowlabel] {$\presheaf{F}$} (E);
    \draw [->] (C) to node [auto,swap,arrowlabel] {$F$} (E);
  \end{tikzpicture}
  \]
This satisfies $\presheaf{F}(f)\isomorphic \colim{f}{F}$.
\end{thm}
Note that in the linear tropical notation the colimit can be written as
 \[
   \bigl(\bigoplus_{d\in D} W(d)\odot J(d)\bigr) (x)
   =
   \bigoplus_{d\in D} W(d)\odot \bigl(J(d)(x)\bigr),
   \]
and this should be thought of as being akin to~\eqref{ddagger}.

The proof of the theorem follows easily from the definitions, given the information in the statement.  The only part of the proof that merits comment here is that one can think, via Lemma~\ref{lem:LeftAdjCocont}, of $\presheaf{F}$ as being cocontinuous by virtue of having a right adjoint $Y\to \presheaf{X}$ given by $y\mapsto \dd_Y\bigl(F({-}),y\bigr)$.

A fundamental point is that the Yoneda embedding is \emph{not}, in general, cocontinuous, so colimits in $X$ do not, in general, get sent to colimits in $\presheaf{X}$, but see below for more on this.   

In this fashion, we can think of the presheaf space as being the ``free cocompletion of $X$''.  Because the enriching category is small, i.e.~because $[0,\infty]$ is a \emph{set} rather just a collection, this can be made into a precise statement (in more general enriching situations there are `size issues').  Let $\GMet$ be the category of generalized metric spaces with short maps and let $\GMetCo$ be the category of cocomplete generalized metric spaces with cocontinuous maps.  Taking the spaces of presheaves gives a functor~ $\presheaf{\cdot}\colon \GMet\to \GMetCo$.  The theorem above implies that this is left adjoint to the forgetful functor $\GMetCo\to \GMet$ which just `forgets' that a cocomplete metric space is cocomplete: for metric spaces $X$ and $Y$
  \[\GMetCo(\presheaf X,Y)\cong \GMet(X,Y).\]

One phrase that category theorists use is ``$\presheaf{X}$ is obtained from $X$ by `freely adjoining all colimits{'}''.  This needs treating carefully, as is  illustrated with the following example.  We can consider a two-point metric space $N_{0,s}$ where the two distances are $0$ and $s>0$.  This is illustrated, along with its embedding in $\presheaf{N_{0,s}}$ in Figure~\ref{Fig:YonedaA0s}.
\begin{figure}
  \[
  \begin{tikzpicture}[scale=1,>=stealth,->,shorten >=2pt, shorten <=2pt,looseness=.5,auto]
    \node [circle,draw=blue,fill=blue,thick, inner sep=0pt,minimum size=1mm ,label=below:$a$](x) at (0,0) {};
  \node [circle,draw=blue,fill=blue,thick, inner sep=0pt,minimum size=1mm ,label=below:$b$](y) at (2,0) {};
  \draw [bend left] (x) to node  {$0$} (y);
  \draw [bend left] (y.west) to node {$s$} (x.east);
  \end{tikzpicture}
  \quad 
  \begin{tikzpicture}[scale=1,baseline=15]
\clip (-1,-1) rectangle (4,4);
\shadedraw[left color=green!30,right color=white, draw=green!50!black, shading angle = 135] (0,0)  -- (5,5) -- (3,5) -- (0,2) -- cycle;
\draw[thin,draw=black,->] (-0.1,0)  -- (4,0) node[below left] {$f(a)$};
\draw[thin,draw=black,->] (0,-0.1) -- (0,4) node[below left] {$f(b)$};
\node [minimum size=0.1cm,inner sep = 0cm] (a) at ( 0,0){};
\draw [draw=blue,fill=blue] (a) circle (0.1cm);
\node [inner sep = 0cm] (b) at ( 0,2){};
\draw [draw=blue,fill=blue] (b) circle (0.1cm);
\node at (0,-0.1) [anchor = north] {$0$};
\node at (-0.1,0) [anchor = east] {$0$};
\node at (0.1,0) [anchor = south west] {$b$};
\node at (-0.1,2) [anchor = east] {$s$};
\node at (0.1,2) [anchor = south west] {$a$};
\end{tikzpicture}
  \]
  \caption{The two-point space $N_{0,s}$ together with its Yoneda image.}
  \label{Fig:YonedaA0s}
\end{figure}
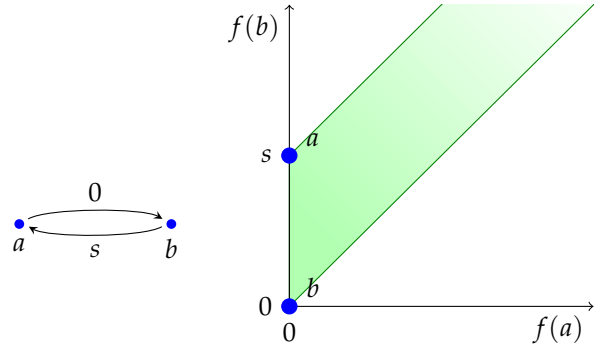
The space $N_{0,s}$ has $a$ as its initial point $\toppoint$, i.e.~as the colimit of the empty diagram.  This colimit is not preserved by the Yoneda embedding, as we have 
  \[\toppoint_{\presheaf{N_{0,s}}}=(\infty,\infty)\not= (0,s)=\Yoneda(\toppoint_{N_{0,s}}),\]
so we can say that this colimit has been freely added.  However, the space $N_{0,s}$ has $b$ as the coproduct $a\sqcup b$, and it is not difficult to calculate that
 \[\Yoneda(a)\sqcup\Yoneda(b)=(0,0)=\Yoneda(b)=\Yoneda(a\sqcup b),\]
so, this colimit has not been `freely added', but rather has been preserved by the Yoneda map.  We will show below that the Isbell completion $I(X)$ is cocomplete and the isometric map $X\to I(X)$  preserves all colimits that exist in $X$.

\section{The Isbell completion is a bicompletion}
In this section we show that for a generalized metric space $X$ the Isbell completion $I(X)$ is both categorically complete and cocomplete, and that the isometry $X\hookrightarrow I(X)$ is both categorically continuous and cocontiuous.  As a corollary we obtain that $I(X)$ has two semi-tropical module structures making it into a metric semi-tropical module and a co-metric semi-tropical module.  This will be illustrated with an example.  We will consider here mainly the colimit versions of these results and just state the corresponding limit versions.  

\subsection{The Isbell completion is cocomplete}
Here we will show that the Isbell completion of any generalized metric space is cocomplete by using the fact that the space of presheaves is cocomplete.

Suppose that $X$ is a generalized metric space, then we have the Isbell adjunction
$L\dashv R$,
and we know by Theorem~\ref{Thm:IsbellAdjunction} that the composite $RL\colon \presheaf X\to \presheaf X$ is idempotent, $RLRL=RL$.  This means that the image of $RL$ is its fixed set $\Fix_{RL}(X)$ and by abuse of notation we can confuse $RL$ with its corestriction to its image $RL\colon \presheaf X\to \Fix_{RL}(X)$.    In the other direction, there is the isometric embedding $\iota\colon \Fix_{RL}(X)\to \presheaf{X}$.  These two maps are adjoint.
\begin{thm}
There is an adjunction $RL\dashv \iota$, so that for all $f\in \presheaf{X}$ and $k\in  \Fix_{RL}(X)$ 
  \[\dd(RL(f),k)=  \dd(f,\iota(k)). \]
\end{thm}
\begin{proof}
We have
\begin{align*}
 \dd(RL(f),k)&=\dd(RL(f),RL(\iota(k)) &&\text{as $k$ is fixed by $RL$}\\
  &\le \dd(f,\iota(k))  &&\text{as $RL$ is a short map}\\
  &\le \dd(f,\iota RL(f))+\dd(\iota RL(f),\iota(k)) & &\text{by the triangle inequality}\\
  &= \dd(L(f),L(f))+\dd(\iota RL(f),\iota(k)) &&\text{by the adjunction $L\leftadjoint R$}\\
  &=\dd(RL(f),k)  &&\text{as $\iota$ is an isometry.}
 \end{align*}
Thus $\dd(RL(f),k)=  \dd(f,\iota(k))  $ as required.
\end{proof}
As $\iota$ is a right adjoint, $\iota$ is categorically continuous and, similarly, as $RL$ is a left adjoint, $RL$ is cocontinuous.  This means that colimits can be calculated in $\Fix_{RL}(X)$ by evaluating in $\presheaf X$ and projecting to $\Fix_{RL}(X)$. 
\begin{thm}
\label{Thm:FixRLComplete}
If $X$ is a generalized metric space then $\Fix_{RL}(X)$ is cocomplete.   For  $J\colon D\rightarrow\Fix_{RL}(X)$  a diagram and $W\colon D^\op\rightarrow [0,\infty]$ a weighting for it, the colimit is calculated by
   \[\colim{W}{J}=RL \colim{W}{\iota J}.\]
\end{thm}
\begin{proof}
We have the following commuting diagram.
  \[
  \begin{tikzpicture}[text height=1.5ex,text depth=.25ex, arrowlabel/.style={font=\footnotesize}]
    \node (D) at (0,0) {$D$};
    \node(FixRL) at (2,0) {$\Fix_{RL}(X)$};
    \node (Xhat) at (4,0) {$\presheaf{X}$};
    \node (FixRL2) at (6,0) {$\Fix_{RL}(X)$};
    \draw [->] (D) to node [auto,arrowlabel] {$J$} (FixRL);
    \draw [->] (FixRL) to node [auto,arrowlabel] {$\iota$} (Xhat);
    \draw [->] (Xhat) to node[auto,arrowlabel]{$RL$} (FixRL2);
    \draw [->]  (FixRL) to [in=225,out=-45] node [auto,arrowlabel] {$\id$} (FixRL2);
  \end{tikzpicture}
  \]
We know by Theorem~\ref{Thm:PresheavesCocomplete} that $\presheaf X$ is cocomplete so $\colim{W}{\iota J}$ exists in $\presheaf X$ and as $RL$ is cocontinuous, we find
 \[RL \colim{W}{\iota J}=\colim{W}{RL\iota J}= \colim{W}{J},\]
 as required.
\end{proof}
As the Isbell completion is isometric to the fixed set of $RL$, and the presheaf space $\presheaf{X}$ is skeletal, we get the following corollary by  Theorem~\ref{Thm:CocompleteCometricModule}.
\begin{cor}
\label{Cor:IsbellCoMetricModule}
If $X$ is a generalized metric space then its Isbell completion $I(X)$ is cocomplete and hence can be given a co-metric semi-tropical module structure.
\end{cor}
We will see that a dual version of this holds, but first we can see the action in action.

\subsection{An example of the co-metric semi-tropical action}
We should now look at the simple example of the asymmetric two-point space $N_{r,s}$.  We know by Section~
\ref{section:IsbellCompletionAsymmetricTwoPoints} that the Isbell completion is just the rectangle  
  \[
   I(N_{r,s})\cong \Fix_{RL}(N_{r,s})=
   \Bigl\{(\alpha,\beta) \bigm | 0\le \alpha \le r,\ 0\le \beta \le s\Bigr\}
   \subset \widehat{N_{r,s}} \subset \R^2_{\text{asym}}.
  \]
We also can easily calculate $RL\colon \widehat{N_{r,s}}\to \Fix_{RL}(N_{r,s})$:
  \[RL\bigl((\alpha,\beta)\bigr)=(\min\{\alpha,r\},\min\{\beta,s\}).\]
The semi-module structure on $\widehat{N_{r,s}}$ 
is simply the following:
\begin{align*}
   \tau\odot(\alpha,\beta)&:= (\tau+\alpha,\tau+\beta);\\
     (\alpha,\beta)\oplus(\alpha',\beta')&:=(\min\{\alpha,\alpha'\},\min\{\beta,\beta'\}).
\end{align*}
Thus the semi-module structure on 
$\Fix_{RL}(N_{r,s})$
 is given by
  \begin{align*}
    \tau\odot(\alpha,\beta)&:= (\min\{\tau+\alpha,r\},\min\{\tau+\beta,s\});\\
     (\alpha,\beta)\oplus(\alpha',\beta')&:=(\min\{\alpha,\alpha'\},\min\{\beta,\beta'\}).
   \end{align*}
 This is pictured in Figure~\ref{Fig:semitropicalArs}.
\begin{figure}[ht]
\begin{center}
\begin{tikzpicture}[scale=1.5]
\draw[thin,draw=black,->] (-0.1,0)  -- (4,0) node[below] {$f(a)$};
\draw[thin,draw=black,->] (0,-0.1) -- (0,3) node[left] {$f(b)$};
\draw[fill=green!80,fill opacity=0.5,thin] (0,0) rectangle (3,2);
\node [circle,draw=blue,fill=blue,thick, inner sep=0pt,minimum size=1mm ,label=above right:$b$](b) at (3,0) {};
\node [circle,draw=blue,fill=blue,thick, inner sep=0pt,minimum size=1mm ,label=above right:$a$](a) at (0,2) {};
\node [outer sep=0pt,inner sep=0pt] (p) at (1,1.2) {};
\node (q) at (2,0.5) {};
\node (pq) at (1,0.5) {};
\node[label=below:$r$] at (b) {};
\node[label=left:$s$] at (a) {};
\node[label=left:$0$] at (0,0) {};
\node[label=below:$0$] at (0,0) {};
\draw[postaction={nomorepostaction,decorate,
                    decoration={markings,mark=at position 0.5 with {\arrow{>}}}
                   },
                   draw=black, very thick] (p) -- +(0.8,0.8);
\draw[postaction={nomorepostaction,decorate,
                    decoration={markings,mark=at position 0.5 with {\arrow{>}}}
                   },
                   draw=black,  very thick] (p)+(0.8,0.8) -- (3,2);
\node [circle,draw=blue,fill=blue,thick, inner sep=0pt,
minimum size=1mm, 
 ,label=above:{$p$}] at (p) {};
\node [circle,draw=blue,fill=blue,thick, inner sep=0pt,minimum size=1mm ,label=below:{$q$}] at (q) {};
\node [circle,draw=blue,fill=blue,thick, inner sep=0pt,minimum size=1mm ,label=below left:{$p \oplus q$}] at (pq) {};
\node (tauodotp) at (2,2.8) {$\tau\odot p$} ;
\draw[>=latex,->](tauodotp) -- (2.1,2.05);
\end{tikzpicture}
\end{center}
\caption{The co-metric semi-tropical module structure on $\Fix_{RL}(N_{r,s})$}
\label{Fig:semitropicalArs}
\end{figure}
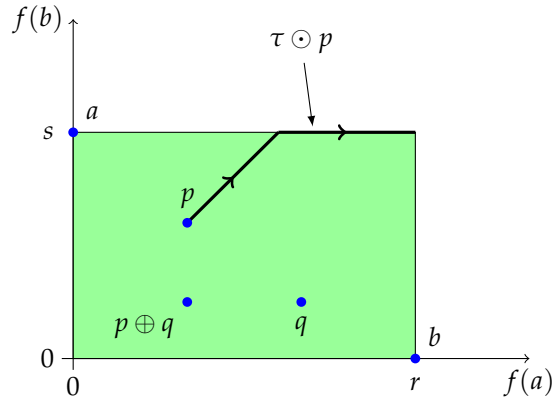

\subsection{The dual results (briefly)}
All of the results about colimits, cocontinuity and cocompletion have analogues in the realm of limits, continuity and completions.  We will present them briefly here.

In the context of weighted limits, as for weighted colimits, 
a \definition{shape} $D$ is a generalized metric space and  a \definition{diagram} of shape $D$ is a short map $J\colon D\to X$.  However, in this context, a \definition{weighting} is a \emph{co-presheaf} $W\colon D\to [0,\infty]$ rather than a presheaf on $W$.  A \definition{limit} of the diagram $J$ with weighting $W$ is an object $l$ of $X$ which satisfies 
  \[
   \dd_X( x,l ) = \sup_{d\in D} \bigl\{\dd_X(x,J(d))\mminus W(d)\bigr\}
   \qquad\text{for all }x\in X.
   \]
   We write $l\isomorphic \lim{W}{J}$.
A short map is \definition{categorically continuous} (or just \definition{continuous}) if $F\lim{W}{J}\isomorphic \lim{W}{FJ}$ for all $W$ and $J$ for which the limit exists.  A generalized metric space is (categorically) \definition{complete} if $\lim{W}{J}$ exists for all $W$ and $J$.

The space of op-co-presheaves $\copre{X}^\op$ is complete and is the free completion in the sense that 
if $F\colon X\to Y$ is a short map to a complete generalized metric space $Y$ then there is a unique-up-to-isomorphism continuous map $\copre{F}\colon \copre{X}^\op\to Y$ such that the $F$ factors through $\copre{F}$ and the co-Yoneda embedding $\Yoneda\colon X\to \copre{X}^\op$; so the following diagram commutes.
 \[
  \begin{tikzpicture}[text height=1.5ex,text depth=.25ex, arrowlabel/.style={font=\footnotesize}]
    \node (C) {$X$};
    \node (Chat) [right of = C] {$\copre{X}^\op$};
    \node (E) [below of = Chat] {$Y$};
    \draw [->] (C) to node [auto,arrowlabel] {$\Yoneda$} (Chat);
    \draw [->] (Chat) to  node [auto,arrowlabel] {$\copre{F}$} (E);
    \draw [->] (C) to node [auto,swap,arrowlabel] {$F$} (E);
  \end{tikzpicture}
  \]
Suppose that $X$ is a generalized metric space then it has a \definition{metric semi-tropical module structure} if it has a commutative monoid structure $\boxplus$ and has a semi-tropical module structure $\boxdot\colon [0,\infty]\times X\to X$ such that for any $x\in X$  we get a morphism of semi-tropical modules
  \[\dd(x,{-})\colon (X,\boxplus) \to ([0,\infty],\max).\]
Analogous to Theorem~\ref{Thm:CocompleteCometricModule}
\begin{thm}
\label{Thm:CompleteMetricModule}
A skeletal generalized metric space is finitely complete if and only if it can be given a metric semi-tropical structure.
\end{thm}
We can now look at the Isbell completion, this time concentrating on the fixed set $\Fix_{LR}(X)\subset \copre{X}^\op$, and seeing that it is complete.  This is the analogue of Theorem~\ref{Thm:FixRLComplete}.
\begin{thm}
\label{Thm:FixLRComplete}
If $X$ is a generalized metric space then $\Fix_{LR}(X)$ is complete.   For  $J\colon D\rightarrow\Fix_{LR}(X)$  a diagram and $W\colon D\rightarrow [0,\infty]$ a weighting for it, the limit is calculated by
   \[\lim{W}{J}=RL \lim{W}{\iota J}.\]
\end{thm}
Combining this with Theorem~\ref{Thm:CompleteMetricModule} we obtain the following.
\begin{cor}
\label{Cor:IsbellMetricModule}
For $X$ a generalized metric space, the Isbell completion $I(X)$ is categorically complete and hence can be equipped with a metric semi-tropical module structure.
\end{cor}
The metric semi-tropical action comes from the action of the semi-tropical semi-ring on $\copre{X}^\op$.  As can be seen from Figures~\ref{Fig:AdIsbellCompletion} and~\ref{Fig:ArsIsbellCompletion} the Isbell completion sits ``the other way round" in $\copre{X}$ from how it does in $\presheaf{X}$ so the metric action goes in ``the other'' direction from how the co-metric action.  This is illustrated in Figure~\ref{Fig:BothModuleStructures}.

\subsection{The Isbell completion is a bicompletion}
We have shown that the Isbell completion $I(X)$ of a generalized metric space $X$ is both complete and cocomplete.  We now show that the canonical isometry $X\to I(X)$ is both continuous and cocontinuous.
%
%
\begin{thm}
\label{Thm:YonedaCompleteCocomplete}
The isometry $X\to I(X)$ of a generalized metric space into its Isbell completion is continuous and cocontinuous.
\end{thm}
\begin{proof}
We will show that the map is continuous; the proof of cocontinuity is similar.  As $I(X)$ is canonically isometric to the fixed set $\Fix_{RL}(X)$ it suffices to proof that the isometry  $\iota_1\colon X \to \Fix_{RL}(X)$ is continuous.  Consider the following commutative diagram.
  \[
  \begin{tikzpicture}[text height=1.5ex,text depth=.25ex, arrowlabel/.style={font=\footnotesize}]
    \node (X) at (0,0) {$X$};
    \node (Xhat) at (2,1) {$\presheaf{X}$};
    \node (FixRL) [below of = Xhat] {$\Fix_{RL}(X)$};
    \draw [->] (X) to node [auto,arrowlabel] {$\Yoneda$} (Xhat);
    \draw [->] (X) to node [auto,arrowlabel] {$\iota_1$} (FixRL);
    \draw [->]  (FixRL) to node [auto,arrowlabel] {$\iota$} (Xhat);
  \end{tikzpicture}
  \]
We know that the Yoneda embedding $\Yoneda$ is continuous, and we know that $\iota$ is continuous by virtue of it being a right adjoint.  Suppose that for $X$ we have a diagram $J$ and a weight $W$ such that the limit $\lim{W}{J}$ exists then  
\begin{align*}
\iota\iota_1\lim{W}{J}&=\Yoneda\lim{W}{J}=\lim{W}{\Yoneda J}
=\lim{W}{\iota\iota_1 J}=\iota\lim{W}{\iota_1 J}.
\end{align*}
So as $\iota$ is an embedding, we have $\iota_1\lim{W}{J}=\lim{W}{\iota_1 J}$ and hence $\iota_1\colon X \to \Fix_{RL}(X)$ is continuous, as required.
\end{proof}

%

\end{document}